\documentclass[11pt]{amsart}
\usepackage{amsmath}
\usepackage{mathrsfs}
\usepackage{amssymb, verbatim}
\usepackage{tikz, tikz-cd}
\usepackage[all,cmtip]{xy}
\usepackage{comment}
\usepackage{color}
\usepackage[scale=0.8,top=2.5cm, bottom=2.5cm,left=2.5cm,right=2.5cm]{geometry}

\title[Collapsing of $ALH^*$-Gravitational Instantons]{Collapsing of $ALH^*$-Gravitational Instantons}

  \author[Y.-S. Lin] {Yu-Shen Lin}
   \email{yslin@bu.edu}
  \address{Department of Mathematics, Boston University, 665 Coommonwealth Ave, Boston, MA 02215}
    \thanks{Y.-S. L. is supported in part by Simons collaboration grant \# 635846 and NSF grant DMS \#2204109.}

\author[R. Takahashi]{Ryosuke Takahashi}
\email{tryotriple@gmail.com}
\address{Department of Mathematics, National Cheng-Kung University , No. 1, Dasyue Rd, East District, Tainan City, Taiwan 701}
\thanks{R. Takahashi is supported by  Ministry of Science
and Technology of Taiwan under grant MOST 111-2636-M-006-023.}

\theoremstyle{plain}
\newtheorem{thm}{Theorem}[section]
\newtheorem{prop}[thm]{Proposition}
\newtheorem{defn}[thm]{Definition}
\newtheorem{lem}[thm]{Lemma}
\newtheorem{cor}[thm]{Corollary}
\newtheorem{conj}[thm]{Conjecture}
\newtheorem{que}[thm]{Question}
\theoremstyle{definition}
\newtheorem{ex}[thm]{Example}
\newtheorem{rk}[thm]{Remark}

\numberwithin{equation}{section}

\newcommand{\T}{\mathcal{T}}

\newcommand{\be}{\begin{equation}}
\newcommand{\bea}{\begin{eqnarray}}
\newcommand{\eea}{\end{eqnarray}}
\newcommand{\ee}{\end{equation}}

\renewcommand{\leq}{\leqslant}
\renewcommand{\geq}{\geqslant}
\renewcommand{\le}{\leqslant}

\renewcommand{\epsilon}{\varepsilon}
\renewcommand{\phi}{\varphi}

\def\Xint#1{\mathchoice
   {\XXint\displaystyle\textstyle{#1}}%
   {\XXint\textstyle\scriptstyle{#1}}%
   {\XXint\scriptstyle\scriptscriptstyle{#1}}%
   {\XXint\scriptscriptstyle\scriptscriptstyle{#1}}%
   \!\int}
\def\XXint#1#2#3{{\setbox0=\hbox{$#1{#2#3}{\int}$}
     \vcenter{\hbox{$#2#3$}}\kern-.5\wd0}}

\def\dashint{\Xint-}
\begin{document}

\begin{abstract}
   By the uniformization theorem of \cite{HSVZ}, any $ALH^*$-gravitational instanton $(X,g)$ can be compactified to a weak del Pezzo surface $Y$ as a complex manifold by adding a smooth anti-canonical divisor $D$. A sequence of $ALH^*$-gravitational instantons $X_i$ converges to a large complex structure limit if the corresponding sequence of weak del Pezzo surfaces $Y_i$ converges to a weak del Pezzo surface $Y_0$ and the corresponding sequence of anti-canonical divisors $D_i$ converges to a nodal curve   
    \cite{CJL2}.
   In this paper, we show that when $(X_i,g_i)$ converges to a large complex structure limit, then $X_i$ with a suitable scaling of the metric $g_i$ pointed Gromov-Hausdorff converges to
    a punctured plane with a special K\"ahler metric. The result can be viewed as a non-compact version of the collapsing result of Gross-Wilson \cite{GW}. Moreover, we provide a partial compactification of the moduli space of pointed $ALH^*$-gravitational instantons with respect to the pointed Gromov-Hausdorff topology. 
	 \end{abstract}
\maketitle
\section{Introduction}
The gravitational instantons are  complete, non-compact hyperK\"ahler $4$-manifolds with $L^2$-bounded curvature. They first appeared in the study of Euclidean quantum gravity introduced by Hawking in 1976 \cite{Haw}. Classified by their volume growth and curvature decay rates, there are exactly six types of gravitational instantons, named as $ALE$, $ALF$, $ALG$, $ALH$, $ALG^*$ and $ALH^*$\cite{SZ}.

In this paper, we focus on $ALH^*$-gravitational instantons which have exotic volume growth $r^{4/3}$. There are two sources of $ALH^*$-gravitational instantons: 
\begin{enumerate}
	\item Let $\check{Y}$ be a weak del Pezzo surface of degree $d$ and $\check{D}\in |-K_{\check{Y}}|$ be a smooth anti-canonical divisor, then $\check{X}=\check{Y}\setminus \check{D}$ admits a Ricci-flat metric asymptotic to the Calabi ansatz \cite{TY}. In particular, $\check{X}$ is an $ALH^*$-gravitational instanton. 
	\item Let $Y$ be a rational elliptic surface and $D$ be an $I_d$-fibre, then ${X}={Y}\setminus {D}$ admits a Ricci-flat metric asymptotic to the standard semi-flat metrics \cite{Hein}. In particular, $X$ is an $ALH^*$-gravitational instanton. 
\end{enumerate}
Hein \cite{Hein} observed that the above two geometry have the same tangent cone, curvature decay, injectivity radius decay and volume growth,...etc. Since hyperK\"ahler rotations do not change the underlying Riemannian metrics, it is natural to expect that after a suitable hyperK\"ahler rotation the former gives the latter. This is proved by the first author and T. Collins and A. Jacob: 
\begin{thm}\cite{CJL} \label{theorem:CJL} Let $\check{X}$ be the complement of a smooth anti-canonical divisor $\check{D}$ in a weak del Pezzo surface $\check{Y}$ of degree $d$ with the Tian-Yau metric $\check{\omega}$. Denote by $\check{\Omega}$ a meromorphic $2$-form on $\check{Y}$ with a simple pole along $\check{D}$ such that $\check{\omega}^2=\frac{1}{2}\check{\Omega}\wedge \bar{\check{\Omega}}$. 
	Given a primitive class $\alpha\in H_1(\check{D},\mathbb{Z})$,
	 then $\check{X}$ admits a special Lagrangian fibration $\check{\pi}:\check{X}\rightarrow \mathbb{R}^2$ with respect to $(\check{\omega},\check{\Omega})$ such that the fibre class is determined by $\alpha$. Moreover, a suitable hyperK\"ahler rotation $X$ with K\"ahler form $\omega$ and holomorphic volume form $\Omega$ given by 
	   \begin{align} \label{HK rel}
	   	   \omega=\mbox{Re}\check{\Omega}, \hspace{4mm} \Omega=\check{\omega}-i\mbox{Im}\check{\Omega},
	   \end{align}
	 admits an elliptic fibration $\pi:X\rightarrow \mathbb{C}$ and it can be compactified to a rational elliptic surface $Y$ by adding an $I_d$-fibre at its end. See the picture below.
	\begin{equation*}
		\begin{tikzcd}[column sep=2em]
			&\check{X}\ar[d,"\check{\pi}"]\ar[bend left,rrr,"\mathrm{HK~rotation}"]& & & 
			{X}\ar[d,"\pi"]\ar[r,hook] & ({Y},\mathrm{I}_{d})\ar[d]\\
			&\mathbb{R}^2 & & & \mathbb{C}\ar[r,hook] & (\mathbf{P}^{1},\infty)
		\end{tikzcd}
	\end{equation*}
\end{thm}
In short, a suitable hyperK\"ahler rotation of $\check{X}$, which is described in (1), becomes $X$, which is the geometry described in (2) as complex manifolds. However, the subtle part is that the metrics from hyperK\"ahler rotation of (1) are not necessarily those asymptotic to the standard semi-flat metrics of Greene-Shapere-Vafa-Yau \cite{GSVY}. To understand the difference of (1) and (2) in the metric level, Collins-Jacob-Lin introduced the notion of generalized semi-flat metrics and proved the following theorem.
\begin{thm} \cite{CJL2}
    With the notation in Theorem \ref{theorem:CJL}, the Ricci-flat metric on $X$ is asymptotic to the generalized semi-flat metric. In particular, there exists a $\mathbb{R}$-family of Ricci-flat metrics on $X$ with those indexed by $\mathbb{Z}$ are the metrics constructed by Hein.  
\end{thm}
In other words, with the introduction of the generalized semi-flat metric, we have the two constructions of $ALH^*$-gravitational instantons are the same. Later Hein-Sun-Viaclovsky-Zhang \cite{HSVZ2} and Collins-Jacob-Lin \cite{CJL3} proved that all the $ALH^*$-gravitational instantons are given by the above constructions. From the classification of weak del Pezzo surfaces or the rational elliptic surfaces with an $I_d$-fibre, there are precisely $10$ diffeomorphism types of $ALH^*$-gravitational instantons. 

For a given diffeomorphism type of gravitational instantons, the cohomology classes of the hyperK\"ahler triples define the so-called period map on the space of marked gravitational instantons (see Definition \ref{def: marked ALH*}), which is the first natural invariant of the gravitational instantons. Collins-Jacob-Lin \cite{CJL3} proved that the cohomology classes of the hyperK\"ahler triples uniquely determine the hyperK\"ahler triples. Lee-Lin later \cite{LL} characterized the image of the period map of marked $ALH^*$-gravitational instantons. These two results together completely characterize the moduli space of $ALH^*$-gravitational instantons. Then the next natural question to ask is:
\begin{que}
	What are the compactification of the moduli spaces of pointed $ALH^*$-gravitational instantons with respect to the pointed Gromov-Hausdorff distance? 
\end{que}	
In the work of Lin-Soundararajan-Zhu \cite{LSZ}, some non-collapsing degenerations of gravitational instantons are studied, and explain how different moduli spaces of gravitational instantons are connected. In this paper, we will study a collapsing degeneration of the $ALH^*$-gravitational instantons. 

The collapsing degeneration of $ALH^*$-gravitational instantons comes from a modified Strominger-Yau-Zaslow conjecture proposed by Kontsevich-Soibelman \cite{KS} and they predicted that Calabi-Yau manifolds converging to the large complex structure limits collapse to affine manifolds with singularities. We first recall the situation of the K3 analogy. Kurikov classified the semi-stable degeneration of K3 surfaces \cite{K} and the central fibre must be one of the following:
\begin{itemize}
	\item (Type I) a smooth K3 surface, or 
	\item (Type II) chain of elliptic ruled surfaces with rational surfaces either ends, or
	\item (Type III) union of rational surfaces such that the double curves on each component is a cycle of rational curves and the dual intersection complex of the central fibre is a triangulation of $S^2$. 
\end{itemize}
When the metrics are rescaled suitably in the sense that the sequence of K3 surfaces have bounded diameter, type I degenerations correspond to non-collapsing degenerations of K3 metrics (with suitable scaling) \cite{A}. Type II degenerations correspond to collapsing to an interval \cite{HSVZ, HSZ}. Type III degenerations, referred to large complex structure limits in the context of mirror symmetry, correspond to the collapsing limit being $2$-dimensional integral affine manifolds with singularities \cite{GW, OO}. In short, the collapsing pictures distinguish the type of semi-stable degenerations for K3 surfaces and collapsing to affine manifolds with singularities can be viewed as the characterization of large complex structure limits from the SYZ viewpoint. Inspired by SYZ mirror symmetry, Collins-Jacob-Lin proposed a definition of the large complex structure limits for pairs of del Pezzo surfaces together with a smooth anti-canonical divisor: a sequence of pairs $(\check{Y}_i,\check{D}_i)$ converges to a large complex structure limit if $\check{Y}_i$ converges to a smooth del Pezzo surface $\check{Y}$ and $\check{D}_i$ converges to $\check{D}\in |-K_{\check{Y}}|$ being nodal and irreducible. The following theorem justifies the notion of the large complex structure limits introduced in \cite{CJL2}. 	
\begin{thm}[see Theorem \ref{thm: compactify 1}(4)] \label{main thm}
	Given a sequence $(\check{Y}_i,\check{D}_i)$ converges to the large complex structure limit and let $\check{X}_i=\check{Y}_i\setminus \check{D}_i$. With a suitable scaling of Tian-Yau metric $\check{g}_i$ on $\check{X}_i$ and $\check{x}_i\in \check{X}_i$ stay in a specific finite region\footnote{See the definition in Section \ref{sec: compactification}.}, then the pointed Gromov-Hausdorff limit of the sequence  $(\check{X}_i, \check{g}_i,\check{x}_i)$ is an affine manifold with singularities and equipped with a Hessian metric. 
\end{thm}
Now we make some comments on the theorem: traditionally, mirror symmetry is expected to  happen at the so-called ``large complex structure limits". In terms of mathematics, the large complex structure limits are understood as the degeneration of Calabi-Yau manifolds into the deepest stratum in the moduli space, for instance the toric degeneration of Calabi-Yau manifolds. The mirror symmetry of pairs consisting of a del Pezzo surface surface and a smooth anti-canonical divisors is first studied in the work of Carl-Pumperla-Siebert \cite{CPS}. The large complex structure limits happen when the 
del Pezzo surfaces admit a toric degeneration and the smooth anticanonical divisors degenerate into an anticanonical cycle. In particular, the definition of the large complex structure limit considered in this paper is not the one used in traditional literature and thus such collapsing result is not expected a priori. 

Another difference of the current result from the work of Gross-Wilson \cite{GW} and Chen-Viaclovsky-Zhang \cite{CVZ} is the geometric setting. They studied the metric behavior for a $1$-parameter family of Calabi-Yau metrics on a fixed manifold with a Calabi-Yau fibration structure when the volume of the fibre degenerates. From the mirror symmetry perspective, one will need to verify that the hyperK\"ahler rotation of a sequence toward the adiabatic limit is converging to the large complex structure limit. However, such hyperK\"ahler rotations are usually non-algebraic and do not fall into an algebraic family for studying mirror symmetry. 
We also point out that due to non-compactness of the ambient spaces, many of the arguments in Gross-Wilson do not go through directly. If one appeals the hyperK\"ahler gluing method in Chen-Viaclovsky-Zhang, then certain Hodge theory statement for the $ALH^*$-gravitational instantons is missing in the literature. Furthermore, hyperK\"ahler gluing does not preserve the cohomology classes of hyperK\"ahler triple and it causes extra difficulties. 



Finally, we can describe part of the boundary of the moduli space of $ALH^*$-gravitational instantons. 
\begin{thm}[=Theorem \ref{thm: compactify 1}]
	There exists a partial compactification of the moduli space of pointed $ALH^*$-gravitational instantons with respect to the pointed Gromov-Hausdorff topology which is locally a polyhedron complex. 
\end{thm}

The arrangement of the paper is as follows: In Section \ref{sec: ALH}, we review the model metrics for $ALH^*$-gravitational instantons, i.e., Calabi ansatz and semi-flat metrics. We then survey the recent development of $ALH^*$-gravitational instantons. In Section \ref{ansatz metric}, we construct an almost Ricci-flat metric on $X$ in Theorem \ref{theorem:CJL} in every cohomology class when $\check{X}$ is close to the large complex structure limit by adapting an argument of Gross-Wilson \cite{GW}. We have global a priori estimates for degenerating metrics qualitatively in Section \ref{sec: a priori} following an argument of Hein \cite{Hein}. Then we use this to prove the collapsing of $ALH^*$-gravitational instantons and the partial compactification of the moduli space of pointed $ALH^*$-gravitational instantons and make a conjecture in Section \ref{sec: compactification}.


\section*{Acknowledgment}
  The first authors would like to thank T. Collins, A. Jacob, T.-J. Lee, S. Sun, R. Zhang for the helpful discussions. The first author would also want to thank National Cheng-Kung University for the hospitality during the visit in January 2023. The key statement is proved during the visit. The first author is supported by Simons collaboration grant \#635846 and NSF grant DMS-2204109. 1 The second author was partially supported by Ministry of
Science and Technology, Taiwan, R.O.C. under Grant no. MOST 111-2636-M-006-023.
\section{$ALH^*$-Gravitational Instantons}   \label{sec: ALH}
\subsection{Local Models} In this section, we will describe the two different models for   neighborhoods of the ends of the $ALH^*$-gravitational instantons. 
 \subsubsection{The Calabi ansatz} \label{sec: Calabi ansatz}
Let $\check{D}$ be an elliptic curve with $K_{\check{D}} \simeq \mathcal{O}_{\check{D}}$ and let $p: E\rightarrow \check{D}$ be a positive line bundle. Denote by $\check{Y}_{\mathcal{C}}$ a tubular neighborhood of the zero section in the total space of $E$ and by $\check{X}_{\mathcal{C}}$ the complement of the zero section in $\check{Y}_{\mathcal{C}}$. Let $\Omega_{\check{D}}$ denote the holomorphic volume form on $\check{D}$. Then $\check{X}_{\mathcal{C}}$ admits a holomorphic volume form given by
\begin{align*}
	\check{\Omega}_{\mathcal{C}}=p^*\Omega_{\check{D}}\wedge \frac{dw}{w},
\end{align*} 
where $w$ is a coordinate along the fibre of $E$. Notice that the single notation $\frac{dw}{w}$ is not well-defined. However, the ambiguity is killed by wedging with $p^*\Omega_{\check{D}}$ and thus the expression $\check{\Omega}_{\mathcal{C}}$ is well-defined. Let $\check{\omega}_D$ be the unique Ricci-flat (actually flat) metric on $\check{D}$ such that $[\check{\omega}_{\check{D}}]=2\pi c_1(E)$, and we normalize $\check{\Omega}_{\check{D}}$ so that
\[
\check{\omega}_{\check{D}} = \frac{\sqrt{-1}}{2}\check{\Omega}_{\check{D}}\wedge \overline{\check{\Omega}}_{\check{D}}.
\]
and let $h$ be a hermitian metric on $E$ such that $\check{\omega}_{\check{D}}=-\sqrt{-1}\partial \bar{\partial}\log h$. Consider the K\"ahler form
\begin{align*}
	\check{\omega}_{\mathcal{C}}=\sqrt{-1}\partial \bar{\partial}\frac{2}{3}\big(-\log{|\xi|^2_h}\big)^{\frac{3}{2}},
\end{align*} 
where $\xi \in E$.  By direct calculation we have
$
\check{\omega}_{\mathcal{C}}^2 = \frac{1}{2} \check{\Omega}_{\mathcal{C}} \wedge \overline{\check{\Omega}}_{\mathcal{C}}.
$
Geometrically, if we let $r(x)$ denote the distance from $x$ to a fixed point $x_0\in \check{X}_{\mathcal{C}}$, then a direct calculation shows that 
\begin{equation}\label{eq: geomTY}
	|\nabla^k Rm|\leq C_k r^{-\frac{k+2}{3}}, \qquad     {\rm inj}(x) \sim r^{-\frac{1}{3}}, \qquad {\rm Vol}(B_{r}(x_0)) \sim r^{\frac{4}{3}}.
\end{equation}In other words, the Riemannian curvature has polynomial decay, but the injectivity radius degenerates at their ends. Thus $\check{X}_{\mathcal{C}}$ is not of ``bounded geometry".  Furthermore, it is easy to check that the metric is complete as $|\xi|\rightarrow 0$ but incomplete as $|\xi|\rightarrow 1$. It is straightforward to see that for each constant $c>0$, the intersection of $\{\xi\in \check{X}_{\mathcal{C}}:|\xi|^2_h=c\}$ with a preimage of a geodesic in $\check{D}$ is a special Lagrangian diffeomorphic to a $2$-torus. In particular, a special Lagrangian fibration in $\check{D}$ respect to $(\check{\omega}_{\check{D}},\check{\Omega}_{\check{D}})$ lifts to a special Lagrangian fibration on $\check{X}_{\mathcal{C}}$ with respect to $(\check{\omega}_{\mathcal{C}},\check{\Omega}_{\mathcal{C}})$.

\subsubsection{Generalized Semi-Flat Metrics}\label{sec: generalized sf metric}
Let $X$ be a holomorphic symplectic manifold with an elliptic fibration $\pi: X\rightarrow B$ over a (possibly non-compact) Riemann surface $B$ which is a submersion with a holomorphic section. Denote by $z$ the parameter of the local coordinate of the base $B$. Then the fibration can be identified with 
  \begin{align*}
  	  T^*B/\mathbb{Z}\tau_1(z)dz\oplus \mathbb{Z}\tau_2(z)dz,
  \end{align*} such that the holomorphic section identified with the zero section via the Abel-Jacobi map. Here $\tau_i$ are multi-valued holomorphic function in $z$ such that $\mathbb{C}/\mathbb{Z}\tau_1(z)\oplus \mathbb{Z}\tau_2(z)\cong \pi^{-1}(z)$. Given a holomorphic volume form $\Omega_{sf}=\kappa(z)dx\wedge dz$, where $x$ is the fibre coordinate and $\kappa(z)$ is a holomorphic function. The semi-flat metric is given by 
 \begin{align} \label{semi-flat metric}
 	  \omega_{sf,\epsilon}:=\frac{i}{2}\big(|\kappa|^2W^{-1}dz\wedge d\bar{z}+W(dx+Bdz)\wedge \overline{(dx+Bdz)}\big), 
 \end{align} where 
  \begin{align} \label{special Kahler}
  	W= \frac{\epsilon}{\mbox{Im}(\bar{\tau}_1\tau_2)} , \hspace{4mm}
  	B(z,x)= -\frac{W}{\epsilon}[\mbox{Im}(\tau_2\bar{x})\partial\tau_1+\mbox{Im}(\bar{\tau}_2x)\partial \tau_2] 
  \end{align} and $\epsilon$ is the volume of elliptic fibres. It is easy to see that ${\omega}_{sf}$ restricts to a flat metric on each fibre and one can check directly that $2{\omega}_{sf}^2={\Omega}_{sf}\wedge \bar{{\Omega}}_{sf}$, i.e. ${\omega}_{sf}$ is a hyperK\"ahler metric. By identifying the section with the base $B$, the first term of \eqref{semi-flat metric}, i.e.  $\frac{i}{2}|\kappa|^2W^{-1}dz\wedge d\bar{z}$, induces a K\"ahler form on the base. It is the so-called special K\"ahler metric associate to the elliptic fibration and we denote it by $\omega_{SK}$. Notice that from the explicit form of the special K\"ahler metric and the Abel-Jacobi map, the scaling of the special K\"ahler metric is intrinsically determined by the holomorphic $2$-form $\Omega$. 

Restricting the situation of a germ of elliptic fibration $Y_{mod}\rightarrow \Delta$ with only an $I_d$-singular fibre over $0\in\Delta$. Locally we may choose $\tau_1=1, \tau_2=\frac{d}{2\pi i}\log{z}$. Take $X_{mod}$ to be the preimage of the punctured disc $\Delta^*$. Then from the previous discussion,
\begin{align*}
	X_{mod}\cong  \Delta^*\times \mathbb{C}/\Lambda(u), \mbox{ where } \Lambda(z)=\mathbb{Z}\oplus \mathbb{Z}\frac{d}{2\pi i}\log{z},
\end{align*} where $z$ is the coordinate on the punctured disc $\Delta^*$. Consider the holomorphic volume form ${\Omega}_{sf}=\frac{1}{z}dx\wedge dz$, then the semi-flat metric \eqref{semi-flat metric} becomes
\begin{align}\label{sf metric}
	{\omega}_{sf,\epsilon} &:=  i \frac{d|\log|z||}{2\pi\epsilon}\frac{dz\wedge d\bar{z}}{|z|^2} \notag \\
	&\quad + \frac{i}{2} \frac{2\pi\epsilon}{d|\log|z||} \left(dx+B(z,x)dz\right)\wedge \overline{\left(dx+B(z,x)dz\right)},
\end{align} where $B(z,x) = -\frac{{\rm Im}(x)}{iz|\log|z||}$.  

It was discovered in \cite[Section 2]{CJL2} that for each $b_0\in \mathbb{R}$, there is a {\it generalized semi-flat metric} given by 
\begin{align*}
	{\omega}_{sf, b_0, \epsilon} &:= i W^{-1}\frac{dz\wedge d\bar{z}}{|z|^2}\\
	&\quad + \frac{i}{2} W  \left(dx+\widetilde{\Gamma}(x,u,b_0) dz\right) \wedge \overline{\left(dx+\widetilde{\Gamma}(z,x,b_0) dz\right)},
\end{align*} where $W= \frac{2\pi \epsilon}{d|\log|z||}$ and $
\widetilde{\Gamma}(z,x,b_0) = B(z,x)+ \frac{b_0}{2\pi^2}\frac{|\log|z||}{z}$. 

 Geometrically, the generalized semi-flat metric can be obtained from the standard semi-flat metric by pulling back along the fibrewise translation map coming from a certain multi-section. Consider a multi-section $\sigma: \Delta^*\rightarrow {X}_{mod}$ of the form 
\begin{align*}
	\sigma: &\Delta^*\longrightarrow {X}_{mod}\\
	& z\mapsto \bigg(z,h(z)+ \frac{c_0}{2\pi i}\log{z}+\frac{b_0}{(2\pi i)^2}(\log{z})^2 \bigg)
\end{align*} for $b_0,c_0\in \mathbb{R}$ and $h(z)$ a holomorphic function on $\Delta^*$. Let $T^*_{\sigma}$ be the fibrewise translation by the multi-section $\sigma$ relative to the zero section. Straightforward calculation shows that $T^*_{\sigma}{\omega}_{sf}={\omega}_{sf,b_0,sf}$ is well-defined on ${X}_{mod}$. It is worth noticing that $2b_0/d\in \mathbb{Z}$ if and only if $\sigma$ extends to a section over $\Delta\rightarrow {Y}_{mod}$, \cite[Lemma 3.28]{FM}. We will give a different interpretation in Section \ref{sec: B-field}.

The following theorem proves that the two local models, Calabi ansatz and the generalized semi-flat metric, are exactly related by hyperK\"ahler rotations. 
    \begin{thm}\cite[Appendix A]{CJL2} \label{thm: local model} Assume that $\check{D}\cong \mathbb{C}/\mathbb{Z}\oplus \mathbb{Z}\tau$ is an elliptic curve, with $\tau$ in the upper half-plane. With the notations in Section \ref{sec: Calabi ansatz}, the special Lagrangian corresponding to $1\in \mathbb{Z}\oplus \mathbb{Z}\tau\cong H_1(\check{D},\mathbb{Z})$ is of phase zero in $\check{X}_{\mathcal{C}}$ with respect to K\"ahler form  ${\omega}_{\mathcal{C}}$ and holomorphic volume form ${\Omega}_{\mathcal{C}}$.  Consider the hyperK\"ahler rotation such that the special Lagrangian fibration becomes an elliptic fibration. Then with a suitable choice of coordinates, one has 
	\begin{align*}
		\mbox{Re}\check{\Omega}_{\mathcal{C}}=\alpha{\omega}_{sf,b_0,\epsilon}, \hspace{5mm} \check{\omega}-i\mbox{Im}\check{\Omega}_{\mathcal{C}}=\alpha{\Omega}_{sf},
	\end{align*} 
	where $b_0=-\frac{1}{2}\mbox{Re}(\tau)d$, $\epsilon=\frac{2\sqrt{2}\pi }{\mbox{Im}(\tau)}$ and $\alpha=\sqrt{d\pi \mbox{Im}(\tau)}$. In particular, there exists a bijection between $\tau \leftrightarrow (b_0,\epsilon)$, i.e., every generalized semi-flat metric can be realized as some hyperK\"ahler rotation of certain Calabi ansatz up to scaling. 
\end{thm}
\subsection{$ALH^*$-Gravitational Instantons}\label{sec: ALH* gi}
  \begin{defn} Let $(X,g)$ be a gravitational instanton with K\"ahler form $\omega$ and holomorphic volume form $\Omega$. 
  	Then $(X,g)$ is of type $ALH_d^*$ if there exists a diffeomorphism $F:X^d_{\mathcal{C}}\rightarrow X$ such that for all $k\in \mathbb{N}$, one has 
  	  \begin{align*}
  	  	  \|\nabla^k_{{ \omega}_{\mathcal{C}}} (F^*\boldsymbol{\omega}-\boldsymbol{\check{\omega}}_{\mathcal{C}})\|_{C^0_g}=O(r^{-k-\epsilon}), 
  	  \end{align*} for some $\epsilon>0$, where $\boldsymbol{\check{\omega}}_{\mathcal{C}}=(\check{\omega}_{\mathcal{C}},\check{\Omega}_{\mathcal{C}})$ and $\boldsymbol{\omega}=(\omega,\Omega)$. 
  \end{defn}
 \begin{rk}
 	 Sun-Zhang proved that the above polynomial decay can actually be improved to exponential decay \cite{SZ}.
 \end{rk}
   There are two types of examples of $ALH^*$-gravitational instantons: 
   The first type of examples comes from the foundational work of Tian-Yau \cite{TY}, where they used the Calabi ansatz $(\check{X}_{\mathcal{C}},\check{\omega}_{\mathcal{C}},\check{\Omega}_{\mathcal{C}})$ to construct complete, non-compact K\"ahler, Ricci-flat manifolds. Recall that a weak del Pezzo surface is a compact complex surface with anti-canonical divisor that is big and nef. From the classification of the compact complex surfaces, weak del Pezzo surfaces are either $\mathbb{P}^1\times \mathbb{P}^1$, the Hirzebruch surface $\mathbb{F}_2$ or the blow up of $\mathbb{P}^2$ at generic $9-d$ points for $d=0,\cdots, 8$.  
   Restrict to dimensional two cases of the construction of Tian-Yau, given a weak del Pezzo surface $\check{Y}$ and $\check{D}$ a smooth anti-canonical divisor, the complement $\check{X}=\check{Y}\setminus \check{D}$ admits a holomorphic volume form $\check{\Omega}$ with a simple pole on $\check{D}$, and $\check{\Omega}$ is unique up to $\mathbb{C}^*$-scaling. A neighborhood of $\check{D}\subset \check{Y}$ can be smoothly identified with a neighborhood of the zero section in the normal bundle to $\check{D}$.  Tian-Yau developed a robust package for solving the complex Monge-Amp\`ere equation on non-compact complex manifolds and applied their techniques to construct complete Ricci-flat metrics $\check{\omega}_{TY}$ on $\check{X}=\check{Y}\setminus \check{D}$ asymptotic to the Calabi ansatz. Hein \cite{Hein} later improved and sharpened the Tian-Yau method, obtaining the result that in a Tian-Yau space the pair $(\check{\omega}_{TY},\check{\Omega})$ converges to $((\Phi^{-1})^*\omega_{\mathcal{C}},(\Phi^{-1})^*\Omega_{\mathcal{C}})$ with exponential decay for a suitable map $\Phi:\check{X}_{\mathcal{C}}\rightarrow \check{X}$ which is a diffeomorphism onto its image \cite{HSVZ}. This is helpful later to transfer geometric properties of $\check{X}_{\mathcal{C}}$ to the geometry properties of the end of $\check{X}$.

   The second example is due to Hein \cite{Hein} and generalized in \cite{CJL2}. Let $Y$ be a rational elliptic surface with an $I_d$-fibre $D$. Denote by $X$ the complement of $D$. 
   This generalized semi-flat metric ${\omega}_{sf,b_0,\epsilon}$ can be used as a model metric near ${D}$ to construct complete hyperK\"ahler metrics on ${X}$ following the arguments of Hein \cite{Hein}. Denote by $[F]\in H_2({X},\mathbb{Z})$ the homology class of the fibre. By choosing a local section near ${D}$, we may identify the elliptic fibration in ${X}$ near its end with the local model ${X}_{mod}$. Consider the long exact sequence of pairs $(Y,X)$, 
     \begin{align} \label{eq: les}
     	  0\cong H_3(Y)\rightarrow H^1(D)\rightarrow H_2(X)\rightarrow H_2(Y)\rightarrow \cdots,
     \end{align} and denote by $[C]$ the generator of the image of $H^1(D)$. Motivated by the Torelli theorem of Looijenga pairs, we say a holomorphic volume from $\Omega$ on $X$ is normalized if its pairing with the image of $H^1(D)$ is $1$. 
   \begin{thm} \cite[Theorem 2.16]{CJL2} \label{generalized Hein's metric} 
   	Let ${\omega}_0$ be a K\"ahler form on $X$ such that 
   	\begin{align*}
   		[{\omega}_0]_{dR}\cdot [F]=\epsilon, \hspace{5mm} [{\omega}_0]_{dR} \cdot [C]=\frac{2b_0}{\epsilon},
   	\end{align*} for some $b_0\in \mathbb{R}$. Then there exists a K\"ahler metric ${\omega}_{\alpha}$ on ${X}$ such that 
   	\begin{align*}
   		{\omega}_{\alpha}= \alpha T^*_{h}{\omega}_{sf,b_0,\frac{\epsilon}{\alpha}}
   	\end{align*} near ${D}$ and ${\omega}_{\alpha}={\omega}_0$ outside of a neighborhood of ${D}$. Here $h$ is a holomorphic section over $\Delta^*$ which depends only on the Bott-Chern cohomology class $[{\omega}_0]_{BC} \in H^{1,1}_{BC}(X, \mathbb{R})$. Moreover, for sufficiently large $\alpha$ \footnote{The condition can be viewed as saying the existence holds sufficiently close to  the large volume limit; see the discussion in \cite[Section 5]{CJL2}.}, there exists a hyperK\"ahler metric on ${X}$,
   	\begin{align*}
   		{\omega}={\omega}_{\alpha}
   		+i\partial \bar{\partial} \phi, 
   	\end{align*} such that 
   	\begin{enumerate}
   		\item  ${\omega}^2=\alpha^2{\Omega}\wedge \bar{{\Omega}}$, and
   		\item ${\omega}$ is asymptotic to the generalized semi-flat metric $T^*_h{\omega}$ in the sense that there exists a constant $\delta>0$ such that 
   		\begin{align*}
   			|\nabla^k \phi|\sim O(e^{-\delta r^{2/3}}),
   		\end{align*} for every $k\in \mathbb{N}$, where $r$ is the distance function to a fixed point in ${X}$. 
   	\end{enumerate}
   	In particular, $\check{\omega}$ and $\check{\omega}_{sf,b_0,\epsilon}$ share the same curvature decay $|\nabla^k Rm|\lesssim r^{-2-k}$ for each $k\in \mathbb{N}$ and injectivity radius decay $inj\sim r(x)^{-1/3}$. 
   \end{thm}
   
   Collins, Jacob and the first author proved that the special Lagrangian fibrations in the model $X_{\mathcal{C}}$ (see Section \ref{sec: Calabi ansatz}) can be deformed to a genuine special Lagrangian fibration near the end in $\check{X}$ with respect to the hyperK\"ahler metric \cite{CJL}. Then a suitable hyperK\"ahler rotation ${X}$ admits an elliptic fibration near the end with monodromy conjugate to $\begin{pmatrix} 1 & d \\ 0 & 1\end{pmatrix}$. By pulling back the universal family over the compactified moduli space of elliptic curves $\bar{M}_{1,1}$, one can compactify ${X}$ to a compact complex surface ${Y}$ with the help of Enrique-Kodaira's classification. Therefore, we will call $\check{X}$ with the hyperK\"ahler triple an $ALH^*_d$-gravitational instanton and one reached the following uniformization result for $ALH^*$-gravitational instantons:
  \begin{thm} \cite{CJL3}\cite{HSVZ2}\label{uniformization}
  	  Any $ALH_b^*$-gravitational instanton $(X,\boldsymbol{\omega})$ admits a suitable\footnote{There are actually infinitely many such hyperK\"ahler rotation, each one corresponds to a primitive element in $\mathbb{Z}^2$.} hyperK\"ahler rotation ${X}$ which can be compactified to a rational elliptic surface ${Y}$ by adding an $I_d$-fibre at its end. Moreover, the hyperK\"ahler metric on ${X}$ is asymptotic to a generalized semi-flat metric with exponential decay. 
  \end{thm} 
  It is also natural to expect another uniformization theorem of $ALH^*$-gravitational instantons:
  \begin{thm} \cite{HSVZ2} \cite{CL} \label{uniformization theorem}
  	 Any $ALH_b^*$ can be compactified to a weak del Pezzo surface of degree $d$ by adding a smooth elliptic curve at its end. 
  \end{thm}
  It is well-known that there are ten deformation families of rational elliptic surfaces together with an $I_d$-fibre, one for each $1\leq b\leq 9, d\neq 8$ and two for $d=8$. As a sanity check, there is a dual statement: there are ten deformation families of weak del Pezzo surfaces (with a smooth anti-canonical divisor) of degree $d$, again one for each $1\leq d\leq 9, d\neq 8$ and two for $d=8$. Thus, a direct consequence of the above uniformization theorems is that there are exactly ten diffeomorphism types of $ALH^*_d$-gravitational instantons and $1\leq d\leq 9$. 
  The uniformization leads to the following definition which compares the cohomology classes of hyperK\"ahler triples on $ALH^*$-gravitational instantons. 
  \begin{defn} \label{def: marked ALH*} 
  	 Fix a diffeomorphism type of $ALH^*$-gravitational instantons and a reference $ALH^*$-gravitational instanton $(X_0,\boldsymbol{\omega}_0)$ with this diffeomorphism type, a marked $ALH^*$-gravitational instanton $(X,\boldsymbol{\omega},\alpha)$ is an $ALH^*$-gravitational instanton $(X,\boldsymbol{\omega})$ with a diffeomorphism $\alpha: X_0\rightarrow X$. 
  \end{defn}
    Then it is natural to ask if the cohomology classes of the hyperK\"ahler triples determine the hyperK\"ahler triple of the $ALH^*$-gravitational instantons and what cohomology classes can be realized as the hyperK\"ahler triple of the $ALH^*$-gravitational instantons for each diffeomorphism type.
   The former uniformization theorem together with a uniqueness theorem of complex Monge-Ampere equation (see \cite[Proposition 4.6]{CJL2}) and a Torelli type theorem for log Calabi-Yau surfaces \cite{GHK} lead to the Torelli type theorem of $ALH^*$-gravitational instantons. 
   \begin{thm}\cite{CJL3}\label{torelli}
   	  Let $(X,\boldsymbol{\omega}), (X',\boldsymbol{\omega}')$ be two $ALH^*$-gravitational instantons of the same diffeomorphism type and $F:X\rightarrow X'$ is a diffeomorphism such that $F^*[\boldsymbol{\omega}']=[\boldsymbol{\omega}]$, then there exists a diffeomorphism $f:X\rightarrow X'$ such that $f^*\boldsymbol{\omega}'=\boldsymbol{\omega}$. 
   	\end{thm} 
   The latter uniformization theorem (Theorem \ref{uniformization theorem}) is used to prove the surjectivity of the period map:
   \begin{thm}\label{surj period}\cite{LL}
   	   Fix a reference underlying space $X$ of an $ALH^*$-gravitational instanton and a triple of cohomology classes $[\boldsymbol{\omega}]\in H^2(X,\mathbb{R})^{\oplus 3}$ which
   	   \begin{enumerate}
   	   	\item does not simultaneously vanish on $(-2)$-classes and 
   	   	\item $[\omega]$ vanishes on the image of $H^1(D)$ in \eqref{eq: les}.
   	   \end{enumerate}
   	    Then there exist an $ALH^*$-gravitational instanton $(X',\boldsymbol{\omega'})$ and a diffeomorphism $f:X\rightarrow X'$ such that $f^*[\boldsymbol{\omega}']=[\boldsymbol{\omega}]$. 
   \end{thm}
  Theorem \ref{torelli} and Theorem \ref{surj period} together characterize the moduli space of the marked $ALH^*$-gravitational instantons. Since we can only talk about convergence with respect to pointed Gromov-Hausdorff topology for non-compact spaces, we say a pointed $ALH^*$-gravitational instanton is a pair $((X,\boldsymbol{\omega},p)$, where $(X,\boldsymbol{\omega})$ is an $ALH^*$-gravitational instanton and $p\in X$ is a point. Two pointed $ALH^*$-gravitational instantons $((X,\boldsymbol{\omega},p),((X',\boldsymbol{\omega}',p')$ are isomorphic if there exists a diffeomorphism $f:x\rightarrow X'$ such that $f^*\boldsymbol{\omega}'=\boldsymbol{\omega}$ and $f(p)=p'$. Then the moduli space of pointed $ALH^*$-gravitational instantons is the set of pointed $ALH^*$-gravitational instantons modulo isomorphisms.   
  To have a more complete understanding of $ALH^*$-gravitational instantons, it is natural to ask the following questions:
  \begin{que}
  		 What is the compactification of the moduli space of pointed $ALH^*$-gravitational instantons with respect to the pointed Gromov-Hausdorff topology? In particular, can we describe the boundaries of the moduli spaces in a geometric meaningful way?   		
  \end{que}
  Given a sequence of pointed $ALH^*$-gravitational instantons with a pointed Gromov-Hausdorff limit, we can first classify the limit by the dimension: the non-collapsing limit if the limit has dimension four, which will be studied in \cite{LSZ}. The collapsing limit if the limit has lower dimension is the main focus of this paper. At the end of the section, we bring the following theorem which motivates this work. 
  \begin{thm} \cite{LLL} \label{LLL} 
  	 Let $(\check{Y}_i,\check{D}_i)$ be a sequence of pairs of del Pezzo surfaces with a smooth anti-canonical divisor such that $\check{Y}_i$ converges to a smooth del Pezzo surface $\check{Y}_0$ and $\check{D}_i$ converge\footnote{If $Y_i\cong \mathbb{P}^1\times \mathbb{P}^1$, then $Y_0\cong \mathbb{P}^1\times \mathbb{P}^1$. Otherwise, $Y_i$ is blow up of $\mathbb{P}^2$ at a generic fixed amount of points. There is a natural topology on the configuration space of points on $\mathbb{P}^2$ and each corresponds to a del Pezzo surface. Here the convergence $Y_i\rightarrow Y_0$, we mean there exist presentations of $Y_i$ as blow up of $\mathbb{P}^2$ and the corresponding configurations converge.} to an irreducible nodal anti-canonical divisor $D_0\subseteq Y_0$. Equip $X_i$ with the Tian-Yau metric which is exact, then the corresponding ${Y}_i$ defined in Theorem \ref{uniformization} (with the hyperK\"ahler rotation such that the lifting of the vanishing cycle of $\check{D}_i$ become holomorphic in ${X}_i$) converge to a rational elliptic surface ${Y}_0$ as a complex manifold. 
  \end{thm}
\begin{rk}\label{rk LLL2}
	The rational elliptic surface $Y_0$ naturally appears in the context of mirror symmetry. Let $D_0$ be the distinguished $I_d$-fibre, where $d=(-K_{\check{Y}_i})^2$ and $X_0=Y_0\setminus D_0$. Then $X_0\rightarrow \mathbb{C}$ is the fibrewise compactification of the Landau-Ginzburg superpotential of the del Pezzo surface of degree $d$ as a monotone symplectic manifold \cite[Theorem 1.4]{LLL}. The construction $Y_0$ for each $d$ as non-toric blow-up of toric surfaces and the singular configuration of $Y_0$ are listed in \cite[Appendix A]{LLL}.
\end{rk}	

 \begin{rk} \label{rk LLL}
 	One can choose the metric on $X_i$ to be the Tian-Yau metric in the K\"ahler class from the restriction of $[\check{\omega}_i]\in H^2(\check{Y}_i)$ such that $[\check{\omega}_i]\rightarrow [\check{\omega}_0]\in H^2(\check{Y}_0)$ is a K\"ahler class, then the conclusion of the theorem still holds with the proof verbatim. 
 \end{rk}
  The above result motivates the following definition: denote by $B_i\cong \mathbb{P}^1$ (and $B_0$) the base of the elliptic fibration on $Y_i$ (and $Y_0$ respectively) and $\Delta_i\subseteq B_i$ the discriminant locus. By choosing $(\check{Y}_i,\check{D}_i)$ and $[\check{\omega}_i]$ generically, we may assume that all of singular fibres of $Y_i$ (except one if $d\neq 1$) are of type $I_1$. Denote by $\tau_i$ the modulus of $\check{D}_i$ such that $|\mbox{Re}\tau_i|\leq \frac{1}{2}, \mbox{Im}\tau_i\geq 1$. Denote by $\omega^i_{SK}$ (and $\omega^0_{SK}$) the special K\"ahler metric defined on $B_i\setminus \Delta_i$ (and $B_0\setminus \Delta_0$ respectively) associate to the elliptic fibration $\pi_i: Y_i\rightarrow B_i$ (and $\pi_0:Y_0\rightarrow B_0$ respectively) and the normalized holomorphic $2$-forms\footnote{See the definition in Section \ref{sec: ALH* gi}.}. There exists diffeomorphisms $f_i:B_0\rightarrow B_i$ such that for every compact set $K\subseteq B_0\setminus \Delta_0$, one has $f_i^*\omega^i_{SK}\rightarrow \omega^0_{SK}$ on $K$. 
  \begin{defn}\label{finite region}
  	 Given a sequence of pointed spaces $ (\check{X}_i, p_i)$, we say $\{p_i\}$ stays in the finite region if there exists a fix compact set $K\subseteq B_0$ avoiding the discriminant locus $\Delta_0$ such that $p_i\in  (f_i\circ \pi_0)^{-1}(K)$ for $i$ large enough. 
  \end{defn}
It worth mentioning that even the homotopy type of $f_i$ is not unique due to the automorphisms of $Y_0$. However, the notion of a sequence $p_i$ stays in the finite region is well-defined.
  From Theorem \ref{thm: local model}, the hyperK\"ahler rotation of the Tian-Yau metric $\check{\omega}_{\mathcal{C}}$ is a scaling of the semi-flat metric such that the fibre over $z$ has diameter bounded above by $O(\mbox{Im}(\tau)^{-1/4}(-\log{|z|})^{1/2})$. When $\mbox{Im}\tau\rightarrow \infty$, the resulting semi-flat metrics on the model geometry pointed Gromov-Hausdorff converge to the base with the special K\"ahler metric. This gives us a hint that the analogue result is possible for $ALH^*$-gravitational instantons.

\section{Construction of Ansatz} \label{sec: ansatz}
     
Given a rational elliptic surface $Y$ with an $I_d$-fibre and all the other singular fibres are of type $I_1$. Denote by $X$ the complement. The goal of this section is to construct an ansatz Ricci-flat metric in each of the K\"ahler class when the area of the elliptic fibre is small. 
 We will refer the readers to Section 3 \cite{GW} for the details of the Ooguri-Vafa metric. 
 
 \subsection{Untwisted Ansatz Metrics} We first recall the following result of gluing the semi-flat metric with Ooguri-Vafa metrics, which is a direct generalization of \cite[Theorem 4.4]{GW}.
\begin{lem}\label{ansatz metric} 
	Let $u_i\in \mathbb{C}$ be in the discriminant locus of the fibration ${X}\rightarrow B=\mathbb{C}$. There exist open neighborhoods $U_i^1\subseteq U_i^2$ of $u_i$ and a K\"ahler metric $\omega^a_{\epsilon}$ with the following properties:
	\begin{enumerate}
		\item The diameter of $U_i^1, U_i^2$ with respect to a fixed metric on the base $B$ is of size $O(\epsilon^p)$, $p<1$.  
		\item $\int_{fibre} \omega^a_{\epsilon}=\epsilon$.
		\item ${\omega}^a_{\epsilon}|_{\mathbb{C}\backslash U^2_i}={\omega}_{sf,\epsilon}$.
		\item ${\omega}^a_{\epsilon}|_{U^1_i}=T_{\sigma_i}^*\omega_{OV}$, here $T_{\sigma_i}$ denotes the translation by a local holomorphic section $\sigma_i$ over $U^{1}_i$ and $\omega_{OV}$ denotes the Ooguri-Vafa metric. 
		\item Let $f_{\epsilon}=\log{\bigg( \frac{\Omega\wedge \bar{\Omega}}{({\omega}^a_{\epsilon})^2}\bigg)}$, then 
		\begin{align*}
			\|f_{\epsilon} \|_{C^2_{g_{sf,\epsilon}}}\le Ce^{-1/C\epsilon^{1-p}},
		\end{align*} with some constant $C$ independent of $\epsilon$. 
	   \item (the integrality condition) $\int_{X} (\omega^a_{\epsilon})^2-\Omega\wedge \bar{\Omega}=0$.
	   \item The diameter of the elliptic fibre is at most $C\epsilon^{\frac{1}{2}}|\log{\epsilon}|$.  
	\end{enumerate}
\end{lem}
For our purpose and the discussion in Section \ref{sec: synergy}, we will apply Lemma \ref{ansatz metric} in the case $p=\frac{1}{2}$ later. 
In particular, Lemma \ref{ansatz metric} implies the following corollary immediately. In the rest of this paper, we use $g^a_{\varepsilon}$ to denote the Riemannian metric corresponding to the K\"ahler form $\omega^a_{\varepsilon}$.
\begin{cor}\label{est_of_f} Let $x_0\in X$ and $d(x):=dist(x,x_0)$ be the distance with respect to $g^a_{\varepsilon}$. There exists $C>0$ such that
\begin{align}
|f_{\varepsilon}|(x)<Ce^{-d(x)^{\frac{2}{3}}/C\varepsilon^{\frac{1}{6}}}\mbox{ with }supp(f_{\varepsilon})\subset A( x_0,1/C\sqrt{\varepsilon}, C/\sqrt{\varepsilon}),
\end{align}
for all $x\in X$ where
\begin{align*}
A( x_0,1/C\sqrt{\varepsilon}, C/\sqrt{\varepsilon}):=\Big\{x\in X\Big|\frac{1}{C}\sqrt{\varepsilon} \leq d(x)\leq C\sqrt{\varepsilon} \Big\}
\end{align*}
is an annular region centered at $x_0$.
\end{cor}
\begin{proof}
By the definition of $\omega_{sf,\varepsilon}$ and (3) in Lemma \ref{ansatz metric}, we have
\begin{align}
\bigg( \frac{\Omega\wedge \bar{\Omega}}{({\omega}^a_{\epsilon})^2}\bigg)=\bigg( \frac{\Omega\wedge \bar{\Omega}}{({\omega}_{sf,\epsilon})^2}\bigg)=\bigg( \frac{\Omega\wedge \bar{\Omega}}{({\omega}_{sf})^2}\bigg)=1
\end{align}
on the complement of $U^2_i$. By (1) in Lemma \ref{ansatz metric}, we can choose $C$ large enough such that $U^2_i\subseteq A( x_0,1/C\sqrt{\varepsilon}, C/\sqrt{\varepsilon})$. So we have
\begin{align}
supp(f_{\varepsilon})\subset A( x_0,1/C\sqrt{\varepsilon}, C/\sqrt{\varepsilon}).
\end{align}
That means Lemma \ref{ansatz metric} (5) implies
\begin{align}
|f_{\varepsilon}|(x)\leq Ce^{-1/C\epsilon^{\frac{1}{2}}}\leq Ce^{-d(x)^{\frac{2}{3}}/C\epsilon^{\frac{1}{6}}},
\end{align}
for all $x\in A( x_0,1/C\sqrt{\varepsilon}, C/\sqrt{\varepsilon})$ and
\begin{align}
|f_{\varepsilon}|(x)=0\leq Ce^{-d(x)^{\frac{2}{3}}/C\epsilon^{\frac{1}{6}}},
\end{align}
for all $x\in X\setminus A( x_0,1/C\sqrt{\varepsilon}, C/\sqrt{\varepsilon})$. So we complete this proof.
\end{proof}
\begin{rk}
In the rest of this paper, we will use $C$ to denote a constant that doesn't depend on $\varepsilon$ and it can be chosen as large as we want. The value of $C$ will increase consecutively
finitely many times in this paper.
\end{rk}

\subsection{Ansatz Twisted by B-Fields} \label{sec: B-field}
The ansatz in Lemma \ref{ansatz metric} has a unique cohomology class after fixing the size of the fibre and thus can not realize all possible cohomology classes of the gravitational instantons. In this section, we will explain how to construct ansatz metrics for other cohomology classes following \cite[Theorem 4.5]{GW}. Recall that $\pi:X\rightarrow B\cong \mathbb{C}$ is the elliptic fibration.  Let $\mathbb{B}\in H^1(B,R^1\pi_*\mathbb{R})$. From the short exact sequence 
  \begin{align} \label{ses}
  	 0\rightarrow R^1\pi_*\mathbb{R} \rightarrow  C^{\infty}(T_{B})\rightarrow  \mathcal{F}\rightarrow 0,
  \end{align} we have $H^0(B,\mathcal{F})\twoheadrightarrow H^1(B,R^1\pi_*\mathbb{R})$ since $C^{\infty}(T_{B})$ is acyclic. Each element in the $H^0(B,\mathcal{F})$ can realized as a Cech cocycle $\{(U_i,\sigma_i)\}_{i\in I}$, where $\{U_i\}_{i\in I} $ is an open cover of $B$ such that 
\begin{enumerate}
	\item each $U_i$ contains at most one critical point of $\pi$,
	\item $\omega^a_{\epsilon}|_{U_i\cap U_j}=\omega_{sf,\epsilon}$ if $i\neq j$.
	\item $\sigma_i\in H^0(U_i,C^{\infty}(T^*_B))$ with $\sigma_i-\sigma_j|_{U_i\cap U_j}\in H^0(U_i\cap U_j, R^1\pi^*\mathbb{R})$. 
\end{enumerate} 
Given such a Cech cocycle, we may construct a twisted holomorphic $2$-form $\mathbb{B}\Omega$ defined by $T_{\sigma_i}^*\Omega_{\epsilon}$ on $U_i$.
since $\sigma_i-\sigma_j$ is a flat section and thus locally holomorphic with respect to $u$ and $\Omega$ locally is of the form $\kappa(u)dx\wedge du$, we have $T_{\sigma_i-\sigma_j}^*\Omega=\Omega$ on $U_i\cap U_j$. Thus, $T_{\sigma_i}^*\Omega=\T_{\sigma_j}^*\Omega$ on $U_i\cap U_j$ and $\mathbb{B}\Omega$ is a globally defined closed $2$-form. Since we still have $\mathbb{B}\Omega\wedge \mathbb{B}\Omega=0$, $\mathbb{B}\Omega$ is locally decomposable and thus determined an integrable complex structure. From \eqref{ses}, any two liftings of an element in $H^1(B,R^1\pi_*\mathbb{R})$ is differed by a global section in $H^0(B,C^{\infty}(T_{B_0}^*))$. Since every smooth section is homotopic to the reference zero section, the cohomology class of $[\mathbb{B}\Omega]$ does not depend on the lifting. 

The following lemma is different from the K3 surface in Gross-Wilson due to the non-compact feature. 
\begin{lem}\label{lem twist}
	There exists a diffeomorphism $F:X\rightarrow X$ isotopic to identity such that $F^*\Omega=\mathbb{B}\Omega$. In other words, the complex manifold with the complex structure determined by $\mathbb{B}\Omega$ is biholomorphic to $X$.
\end{lem}
\begin{proof}
	Denote by $X_0$ the complement of singular fibres in $X$. From the long exact sequence of relative pairs and Poincar\'e duality (for instance, see \cite[Section 7.12]{D}), we have 
	  \begin{align}
	  	H_3(X)\rightarrow H^1(X\setminus X_0) \rightarrow H_2(X_0)\rightarrow H_2(X)\rightarrow H^2(X\setminus X_0) \xrightarrow{\delta} H_1(X_0)\rightarrow 0.
	  \end{align} Here we use the coefficients in $\mathbb{R}$ and we have $H_1(X)\cong 0$ from \cite[Corollary 7.5]{SZ}. Since $X\setminus X_0$ is the disjoint union of the singular fibres, we have $\mbox{dim}_{\mathbb{R}}H^1(X\setminus X_0)=\mbox{dim}_{\mathbb{R}}H^2(X\setminus X_0)=k$, where $k$ is the number of singular fibres as we assume all the singular fires are of type $I_1$. The Leray spectral sequence for the fibration $\pi:X_0\rightarrow B_0$,
        \begin{align*}
        	  H_p(B_0, H_q(T^2))\Rightarrow H_{p+q}(X_0,\mathbb{R}),
        \end{align*} degenerates at $E_2$-page. We have $\mbox{dim}_{\mathbb{R}}H_1(X_0)=k$ and thus $\delta$ is an isomorphism. In particular,  $H_2(X_0)\twoheadrightarrow H_2(X)$. Again from the Leray spectral sequence of the fibration, the $2$-cycles of $X$ are generated by the fibre or $S^1$-fibration over graphs in the base $B$. 	
	Direct calculation shows that $\Omega-\mathbb{B}\Omega$ is a pull-back of $2$-form from the base. Since any element $H_2(X)$ can be realized as some $2$-cycle with projection to $1$-skeleton under $\pi$, the periods of $\Omega$ and $\mathbb{B}\Omega$ are exactly the same. 
	Then it follows from the Torelli theorem of log Calabi-Yau surfaces \cite{GHK}. 
\end{proof}

We may similarly define $\mathbb{B}\omega$:
since $\sigma_i-\sigma_j$ is a flat section, direct calculation shows that  $T_{\sigma_i-\sigma_j}^*\omega_{\epsilon}^a=\omega_{\epsilon}^a$ on $U_i\cap U_j$ because $\omega_{\epsilon}^a|_{U_i\cap U_j}$ is the semi-flat metric. Equivalently, we have $T^*_{\sigma_i}\omega_{\epsilon}^a=T^*_{\epsilon_j}\omega_{\epsilon}^a$ on $U_i\cap U_j$. Thus, $\mathbb{B}\omega$ is a well-defined closed K\"ahler form with respect to the almost complex structure given by $\mathbb{B}\Omega$. From Lemma \ref{lem twist}, we have $(F^{-1})^*\mathbb{B}\omega$ is a K\"ahler form on $X$ in the cohomology class $[\mathbb{B}\omega]$. 

\begin{ex}
	We apply the construction to the semi-flat metric $\omega_{sf,\epsilon}$ on $X_{mod}$ in \eqref{sf metric}. From the monodromy of the fibration of $X_{mod}\rightarrow \Delta^*$, we have $H^1(\Delta^*,R^1(\pi_{mod})_*\mathbb{R})\cong \mathbb{R}$. 
	For any $\mathbb{B}\in H^1(\Delta^*,R^1(\pi_{mod})_*\mathbb{R})$, we may represent it by simply connected open cover of $\Delta^*$ with the holomorphic function to be $b_0\log{z}^2$, $b_0\in \mathbb{R}$. Then we have $\mathbb{B}\Omega=\Omega$ and one gets the generalized semi-flat metrics in Section \ref{sec: generalized sf metric}. 
\end{ex}
 Moreover, we may achieve all possible K\"ahler classes this way. 
\begin{lem} \cite[p.45]{GW} \label{B-twist Kahler form}
	Given any K\"ahler class $[\omega]\in H^2(X,\mathbb{R})$ with size of the elliptic fibre being $\epsilon$, then there exists $\mathbb{B}\in H^1(B,R^1\pi_*\mathbb{R})$ such that $[\mathbb{B}\omega_{\epsilon}^a]=[\omega]$. Moreover, 
	  \begin{align*}
	  	  [\mathbb{B}\omega_{\epsilon}^a]-[\omega_{\epsilon}^a]=\epsilon \mathbb{B}.
	  \end{align*}
\end{lem}
\begin{rk}
	In Lemma \ref{lem twist} and Lemma \ref{B-twist Kahler form}, the difference from the K3 case is that $H_2(B,H_0(T^2))\cong 0$ in our situation and that provides the only $2$-cycles possibly has non-trivial integration against any $2$-form pull back from the base. 
\end{rk}
Combining Lemma \ref{ansatz metric} and Lemma \ref{B-twist Kahler form}, we have 
\begin{lem} \label{ansatz metric 2}
	Given any K\"ahler class $[\omega]\in H^2(X,\mathbb{R})$ with the size of elliptic fibre $\epsilon$, there exists a K\"ahler form $\omega_{\epsilon}^a$ in this K\"ahler class such that the properties (1)-(6) in Lemma \ref{ansatz metric} are true. 
\end{lem}

We recall the following definition for a certain type of non-compact geometry at its end, which will be used in the proof of Theorem \ref{c2alpha}.  
\begin{defn}\label{CYL}
	A complete Riemannian manifold $(X,g)$ is $\mbox{CYL}(\gamma)$ with $\gamma\in [0,1)$ if there exists $x_0\in X$ and $C\geq 1$ such that
	\begin{enumerate}
		\item  $|B(x_0,s)|\leq Cs^2$ for all $s\geq C$.
		\item  $\mbox{Ric}(x)\geq Cr(x)^{-2\gamma}$ for all $x\in X$ with $r(x)\geq C$.
		\item $A(x_0, r(x)-\frac{1}{C}r(x)^{\gamma}, r(x)+\frac{1}{C}r(x)^{\gamma})\subseteq B(x,Cr(x)^{\gamma})$, for all $x\in X$ with $r(x)\geq C$.
	\end{enumerate}
Here $A( x_0,r, s):=\Big\{x\in X\Big|r\leq r(x)\leq s \Big\}$. We will call $C$ a $CYL(\gamma)$ constant of $X$. 
\end{defn}
 Let $(Y_0,D_0)$ be a pair consisting $Y_0$ a rational elliptic surface with an anti-canonical cycle $D_0$ such that $X_0=Y_0\setminus D_0$ has trivial periods. Let $(Y,D)$ be a small deformation of $(Y_0,D_0)$ and $X=Y\setminus D$. 
	It is known that $X$ with the Riemannian metric $g^a_{\epsilon}$ induced by $\omega_{\epsilon}^a$ and any $ALH^*$ gravitational instantons are $\mbox{CYL}(\frac{1}{3})$ \cite[p. 379]{Hein}. Moreover, the constant in Definition \ref{CYL} for $(X,g^a_{\epsilon})$ can be chosen independent of the small deformation $(Y,D)$ from the lemma below.
\begin{lem}
	Let $Y$ be a rational elliptic surface, $D$ be an $I_d$-fibre and $X=Y\setminus D$. Let $g_{\epsilon}^a$ be the Riemannian metric induced by $\omega^{a}_{\epsilon}$ constructed in Lemma \ref{ansatz metric}. Then $(X,cg^a_{\epsilon})$ is $CYL(\frac{1}{3})$ with a $CYL(\frac{1}{3})$ constant of order $O(c)$. 
\end{lem}
\begin{proof}
	This is proved by Hein \cite[p.379]{Hein} but we include here for the readers' convenience and for the quantitative estimates. We will identify the base and the section and thus the base admits a K\"ahler metric which is the restriction of the semi-flat metric to the section. Notice that this K\"ahler metric is a multiple of the special K\"ahler metric discussed in Section \ref{sec: generalized sf metric}. Fix a point $x_0$ on the section. 
	For $|z|\ll 1$, the distance from $x_0$ to a point $x$ on the section is $O(c^{\frac{1}{2}}\epsilon^{-\frac{1}{2}}|\log{|z|}|^{\frac{3}{2}})$ and the length of the Euclidean circle of radius $|z|$ is $O(c^{\frac{1}{2}}\epsilon^{-\frac{1}{2}}|\log{|z|}|^{\frac{1}{2}})$. Together with the diameter of an elliptic fibre is $O(c^{\frac{1}{2}}\epsilon^{\frac{1}{2}}|log{|z|}|^{\frac{1}{2}})$.	
	This implies that there exists a uniform constant satisfies Definition \ref{CYL} (3) when we choose the $CYL$-constant at least of order $O(1)$. Because $g^a_{\epsilon}$ is Ricci-flat except at the gluing region, Definition \ref{CYL} (2) holds automatically for the $CYL$-constant at least of order $O(1)$. 
	
Finally, let $K$ be the set of points in the base with distance to $x_0$ is at most $r$ within the section. Then by Fubini's theorem its volume with respect to $cg^a_{\epsilon}$ is 
	 \begin{align*}
	 	 \mbox{Vol}(\pi^{-1}(K),cg^a_{\epsilon})&=\mbox{Vol}(\pi^{-1}(pt),cg^a_{\epsilon})\cdot \mbox{Vol}(K,cg^{a}_{\epsilon})\\
	 	                                        &=(c\epsilon)\cdot \mbox{Vol}(K,cg^{a}_{\epsilon})\\
	 	                                        &=(c\epsilon) (c\epsilon^{-1} \epsilon_0) \mbox{Vol}(K, g_{\epsilon_0}^a),
	 \end{align*}  for any fixed $\epsilon_0\ll 1$. Again using the fact that the diameter of an elliptic fibre is $O(c^{\frac{1}{2}}\epsilon^{\frac{1}{2}}|log{|z|}|^{\frac{1}{2}})$, we have
    \begin{align*}
    	\big|B^{\epsilon g^a_{\epsilon}}(x_0,r)\big|&\lesssim O\bigg(c^2 \int_{|z_x|} \log{\frac{1}{t}}\frac{dt}{t}\bigg)\\
    	&=O(c^2 |\log{|z_x|}|^2)=O\bigg(c^{\frac{4}{3}}\varepsilon^{\frac{2}{3}} \big(c^{\frac{1}{2}}\varepsilon^{-\frac{1}{2}}|\log{|z_x|}|^\frac{3}{2}\big)^{\frac{4}{3}} \bigg)= O(c^{\frac{4}{3}}\varepsilon^{\frac{2}{3}} r^{\frac{4}{3}})\lesssim O(c\varepsilon r^2),
    \end{align*} for some universal constant $C$ independent of $X,c,\epsilon$. Here the last equality comes from the fact that  $dist_{cg^a_{\epsilon}}(x_0,z)=O(c^{\frac{1}{2}}\epsilon^{-\frac{1}{2}}|log{|z_x|}|^{\frac{3}{2}})$. This finishes the proof of the lemma.

\end{proof}

\subsection{Synergy} \label{sec: synergy}
	For cases other than $I_1$-fibres, one can use the local models of Chen-Viaclovsky-Zhang \cite{CVZ} to derive analogue result in Theorem \ref{main thm}. Instead, we will understand how fast the singular fibres of $Y_i$ collide from the work of \cite{LLL}. 
	We first recall some of the asymptotics from \cite{CJL} and \cite{LLL} which is helpful for the later analysis under the setting and notations in Theorem \ref{LLL}: consider the Tian-Yau metric $\check{\omega}_i$ on $\check{X}_i=\check{Y}_i\setminus \check{D}_i$ with the corresponding Calabi-Yau metric on $\check{D}_i$ has volume $1$. Assume that the special Lagrangian torus with smallest volume in $\check{X}_i$ has volume of order $O(|\log{|t_i|}|^{-1/2})$, then all the special Lagrangian torus in $\check{X}_i$ of other homology classes has volume of order $O(|\log{|t_i|}|^{1/2})$, when $|t_i|\ll 1$.\footnote{For instance, the assumption holds when there exists a complex degeneration with $t$ being the complex parameter.} Here $O(|\log{|t_i|}|^{-1/2})$ denotes the collection of functions bounded between $C^{-1}|\log{|t_i|}|^{-1/2}$ and $C|\log{|t_i|}|^{-1/2}$ for some constant $C>0$ independent of $\check{X}_i$. Under this setting, one can find representatives $\gamma_{i,j}\in H_2(\check{X}_i,\mathbb{Z})$ such that $\gamma_{i,j}$ generate $H_2(\check{X}_i)/\mbox{Im}(H^1(\check{D}_i))$ and 
	\begin{align} \label{other periods}
		\bigg|\int_{\gamma_{i,j}}\check{\Omega}_i\bigg|=O(1),
	\end{align} where $\check{\Omega}_i$ is the meromorphic volume form on $\check{Y}_i$ with a simple pole along $\check{D}_i$ such that $2\check{\omega}_i^2=\check{\Omega}_i\wedge \bar{\check{\Omega}}_i$. Define $\omega_i,\Omega_i$ from $\check{\omega},\check{\Omega}_i$ from the hyperK\"ahler relation \eqref{HK rel}. Then $|\log{t_i}|^{-3/2}\check{\omega}_i$ after the hyperK\"ahler rotation is asymptotic to a scaling of the semi-flat metric $\epsilon_i\omega_{sf,b_i,\epsilon_i}$ from Theorem \ref{local model} for $\epsilon_i\sim O(|\log{|t_i|}|^{-1})$ and some $b_i\in \mathbb{R}$. 
	
	We may assume that $Y_i$ has only $I_1$-singular fibres but $Y_0$ may have more degenerate singular fibres. This can be achieved if $(\check{Y}_i,\check{D}_i)$ and cohomology classes $[\check{\omega}_i]$ are chosen generically. 
	Then by Lemma \ref{ansatz metric} and the discussion in Section \ref{sec: B-field}, there exist K\"ahler forms $\omega_{i,\epsilon_i}^a$ from Lemma \ref{ansatz metric} asymptotic to $\omega_{sf,b_i,\epsilon_i}$ 
	Notice that $\epsilon\epsilon_i\omega_{i, \epsilon\epsilon_i}^a$ converges to the special K\"ahler form $\omega_{SK,i}$ on the base of $X_i$, as $\epsilon\rightarrow 0$. The special K\"ahler metric $\omega_{SK,i}$ has singularity at the discriminant locus and the metric is incomplete at the singularities. Because $Y_i\rightarrow Y_0$, the discriminant locus of the base of $X_i$ converge. 
	Some of the points of discriminant locus, say $y_1,\cdots, y_l$ of the base of $X_i$ may converge to the same point, say $y_*$ in the discriminant locus of the base of $X_0$. Again by choosing $(\check{Y}_i,\check{D}_i, [\check{\omega}_i])$ generic, we may assume that the distance between $y_i$ with respect to $\omega_{SK,i}$ are of the same order $d_i$. 
	
	Denote by $\gamma_{i,j}$ the parallel transport of a component of a fibre over $y_*$ in $X_0$ to $X_i$. Observe that the Leray spectral sequence of the special Lagrangian fibration on $X_i$ degenerate at the $E_2$-page. In particular, we can choose the representative of $\gamma_{i,j}$ such that they are generically $S^1$-fibrations over graphs in the base.   
	Then from \eqref{other periods} we have  
	\begin{align*}
		O(\epsilon_i^{\frac{3}{2}})=\bigg|\int_{\gamma_{i,j}}|\log{|t_i|}|^{-\frac{3}{2}}\check{\Omega}_i\bigg|=O\bigg(\big|\int_{\gamma_{i,j}}\epsilon_i \omega_{i}\big|\bigg)=O\bigg(\big|\int_{\gamma_{i,j}}\epsilon_i \mathbb{B}_i\omega^a_{i,\epsilon_i}\big|\bigg).
	\end{align*} Here the second equality comes from \eqref{HK rel} and choosing $(\check{Y}_i,\check{D}_i)$ generic. The last equality holds because $[\omega_i]=[\mathbb{B}_i\omega^a_{i,\epsilon_i}]$. Observe that 
	\begin{align*}
		\int_{\gamma_{i,j}}\mathbb{B}_i\omega^a_{i,\epsilon_i}-\int_{\gamma_{i,j}}\omega^a_{i,\epsilon_i}=O\bigg( \int_{T^2}\omega^a_{i,\epsilon_i}\bigg)=O(\epsilon_i),
	\end{align*} where $T^2$ denotes the homology class of the elliptic fibre of $X_i$.  This shows that 
$\big| \int_{\gamma_{i,j}}\varepsilon_i\omega^a_{\varepsilon_i}\big|=O(\varepsilon_i^{\frac{3}{2}})$. 
Notice that the diameter of the fibre in $X_i$ with respect to $\epsilon_i\omega^a_{i,\epsilon_i}$ is of order $O(\epsilon_i)$, the Fubini theorem implies that $\int_{\gamma_{i,j}}\epsilon_i\omega^a_{i,\epsilon_i}=O(d_i\epsilon_i)$ and thus we have $d_i \lesssim O(\epsilon_i^{\frac{1}{2}})$. In particular, when we glue the Ooguri-Vafa metric with the semi-flat metric in Lemma \ref{ansatz metric}, we have to take $p\geq \frac{1}{2}$ so that the gluing region for each singular fibre can be chosen non-overlapping.

\section{A Priori Estimates} \label{sec: a priori}
Let $Y_0$ be a rational elliptic surface with an $I_d$-fibre $D_0$ as in Theorem \ref{LLL}. Let $(Y,D)$ be a pair consist of a rational elliptic surface $Y$ and an anti-canonical cycle $D$ which is a small deformation of the pair $(Y_0,D_0)$ such that $D\cong D_0$. Denote $Y\setminus D$ by $X$. 
Let $\omega^a_{\varepsilon}$ be the ansatz given by Lemma \ref{ansatz metric 2} with respect to the normalized holomorphic volume form on $X$.  
From the Torelli theorem of $ALH^*$-gravitational instantons (Theorem \ref{torelli}), there exists a unique Ricci-flat metric in the cohomology class $[\mathbb{B}\omega_{\epsilon}^a]$ for a given $\mathbb{B}$ and we denote it by $\omega_{\epsilon}$. We will prove the $C^{2,\alpha}$-estimate for the unique $\omega_{\varepsilon}$ in this section. Let $\Omega$ be the normalized holomorphic volume form on $X$. To obtain such an estimate, we shall first write $\omega_{\varepsilon}=\omega^a_{\varepsilon}+i\partial\bar{\partial}u_{\varepsilon}$ for some real value function $u_{\varepsilon}$ defined on $X$ by using $\partial\bar{\partial}$-lemma. Then $u_\varepsilon$ will be the solution of Monge-Amp\`ere equation
\begin{align}\label{MAeq}
(\omega^a_{\varepsilon}+i\partial\bar{\partial}u_{\varepsilon})^2=e^{f_{\varepsilon}}(\omega^a_{\varepsilon})^2
\end{align}
for the real-valued function 
\begin{align*}
f_{\varepsilon}=\log \bigg( \frac{\Omega\wedge\bar{\Omega}}{2(\omega^a_{\varepsilon})^2}\bigg).
\end{align*}
One shall recall that $f_{\varepsilon}$ satisfies the properties given by Lemma \ref{ansatz metric} and Corollary \ref{est_of_f}.
For a fix $\alpha\in (0,1)$, our goal in this section is to derive the following estimate, which is the key step toward Theorem \ref{main thm}. Notice that the argument for the analogue result in Gross-Wilson does not apply here. 
\begin{thm}\label{c2alpha} 
There exist $\epsilon_0>0$ and $x_0\in X$ with the following significance. For each $R>0$, there exists $C_{\epsilon_0}	>0$ independent of deformation of $X$ such that
\begin{align}
\Big\|u_{\varepsilon}-\dashint_{B^{d^{\varepsilon_0}}(x_0,R)}u_{\varepsilon}\ \Big\|_{C^{2,\alpha}_{g_{\varepsilon}^a}(\overline{B^{d^{\varepsilon_0}}(x_0,R)})}\leq C_{{\varepsilon_0}}R\varepsilon^{\frac{1}{6}},
\end{align}
for all $\varepsilon<\epsilon_0$, where  $d^{\varepsilon_0}(x,y)$ be the distance function defined by $g^a_{\varepsilon_0}$. Moreover, the constant $C_{\varepsilon_0}$ depends on $\|f_{\epsilon}\|_{C^2(g^a_{\epsilon})}$ and $CYL(\frac{1}{3})$-constant of $X$. 
\end{thm}
A direct consequence of the theorem is as follows:
\begin{cor} \label{cor: limit}
	With the above setting and notations, we have 
	     \begin{align}
	     	\| \omega_{\epsilon} -\mathbb{B}\omega_{\epsilon}^a\|_{C^{0,\alpha}_{g_{\varepsilon}^a}(\overline{B^{d^{\varepsilon_0}}(x_0,R)})}\leq C_{\epsilon_0} R\varepsilon^{\frac{1}{6}},
	     \end{align} where the constant $C_{\epsilon_0}$ is the one in Theorem \ref{c2alpha}.
\end{cor}

To prove this theorem, we need some general fact about the Monge-Amp\`ere equations. First of all, by \cite[Theorem 1.1]{TY}, we have
\begin{align}\label{TY_est}
\|u_{\varepsilon}\|_{C^2_{g_{\varepsilon}^a}}\leq C,
\end{align}
for some $C$ only depends on the constant in Lemma \ref{ansatz metric}(5). In particular, it is independent of the choice of $\varepsilon$. This is because $\|f_{\varepsilon}\|_{C^2_{g_{\varepsilon}^a}}$ is bounded by some universal constant according to Corollary \ref{est_of_f} (In fact, we only need $C^1$-uniform bound for $f_{\varepsilon}$ here). With this estimate (\ref{TY_est}) in our mind, we show that (\ref{MAeq}) is a uniformly elliptic PDE in the following paragraphs.

Fix any point $p\in X$ and by choosing a local normal coordinate at $p$, we have 
\begin{align}
(\omega^a_{\varepsilon}+i\partial\bar{\partial}u_{\varepsilon})^2&=\Big[1+\partial_z\bar{\partial}_zu_{\varepsilon}+\partial_w\bar{\partial}_wu_{\varepsilon}+\Big(\partial_z\bar{\partial}_zu_{\varepsilon}\partial_w\bar{\partial}_wu_{\varepsilon}-\partial_z\bar{\partial}_wu_{\varepsilon}\partial_w\bar{\partial}_zu_{\varepsilon} \Big)\Big]\\
&=A^{ij}(\nabla^2u_{\varepsilon})\partial_i\bar{\partial}_j u_{\varepsilon}+1=e^{f_{\varepsilon}}\nonumber
\end{align}
where $i,j=1,2$, $\partial_1:=\partial_z$, $\partial_2:=\partial_w$ and
\begin{align}
A(\nabla^2u_{\varepsilon})=
\left(
\begin{array}{cc}
1+\frac{1}{2}\partial_w\bar{\partial}_wu_{\varepsilon} & -\frac{1}{2}\partial_z\bar{\partial}_wu_{\varepsilon}\\
-\frac{1}{2}\partial_w\bar{\partial}_zu_{\varepsilon} & 1+\frac{1}{2}\partial_z\bar{\partial}_z u_{\varepsilon}
\end{array}
\right).
\end{align}
We assume the eigenvalues of $A$ are $\lambda,\Lambda$ with $\lambda\leq \Lambda$. Meanwhile, we have
\begin{align}
(\omega^a_{\varepsilon}+i\partial\bar{\partial}u_{\varepsilon})^2=2\det \left(
\begin{array}{cc}
1+\partial_w\bar{\partial}_{w}u_{\varepsilon}& -\partial_z\bar{\partial}_wu_{\varepsilon}\\
-\partial_w\bar{\partial}_zu_{\varepsilon} &1+\partial_z\bar{\partial}_{z}u_{\varepsilon}
\end{array}
\right)
:=2\det (\dot{A}).
\end{align}
Suppose $\dot{A}$ has eigenvalues $\dot{\lambda}$, $\dot{\Lambda}$ where $\dot{\lambda}\leq\dot{\Lambda}$. Then we have
\begin{enumerate}
\item $\dot{\lambda}\dot{\Lambda}=2\det(\dot{A})=e^{f_{\varepsilon}}=O(1)$,
\item $\sup_{|v|=1} v^*\dot{A}v\leq C\|\nabla^2u_{\varepsilon}\|_{C^0}\leq C$.
\end{enumerate}
These inequalities imply that
\begin{align}\label{ellp}
\frac{1}{C}\leq \dot{\lambda}\leq v^*\dot{A}v \leq \dot{\Lambda}\leq C,
\end{align}
for all $v\in \mathbb{C}^2$, $|v|=1$. Meanwhile, we can easily check that $\dot{A}+Id=2A$. So (\ref{ellp}) implies
\begin{align}
\frac{1}{C}+1\leq v^*\dot{A}v+1=2v^*Av \leq C+1,
\end{align}
which implies
\begin{align}
\frac{1}{C}+\frac{1}{2}\leq v^*Av \leq C+\frac{1}{2},
\end{align}
for all $v\in \mathbb{C}^2$, $|v|=1$. Therefore we have
\begin{align}\label{u_ellip}
\frac{1}{2}\leq \lambda\leq \Lambda\leq C+\frac{1}{2}.
\end{align}
So we have
\begin{align}\label{L_over_l}
\frac{\Lambda}{\lambda}\leq C+1
\end{align}
for some $C$ which is independent of $\varepsilon$. In addition, by assuming a priori we have the smooth solution $u$ and estimate of $|\nabla^2 u|$, we can conclude that $u$ is a solution of the following linear equation
\begin{align}\label{eq_1}
L(u):=A_{ij}\partial_i\bar{\partial}_ju=e^{f_{\varepsilon}}-1.
\end{align}
This is a uniform elliptic PDE because of (\ref{u_ellip}).

Now, the following Proposition is proved based on the argument developed by Hein \cite[Proposition 2.9]{Hein} incorporating with the degeneracy of the metrics respect to $\omega^a_\epsilon$.
\begin{prop}\label{lema_Hein}
There exists $C>0$ which is independent of the choice of $\varepsilon$ such that
\begin{align}
\int_{X}|\nabla u_{\varepsilon}|^2\leq C\varepsilon^{\frac{4}{3}},
\end{align}
for any $\varepsilon$ small.
\end{prop}
\begin{proof}To make our notation simpler, we write $u_{\varepsilon}=u$ and $\|\cdot\|_{C^k}=\|\cdot\|_{C^k_{g^a_{\varepsilon}}}$, for all $k\in\mathbb{N}_0$ in this proof.
Here we separate our proof into three steps.\\
{\bf Step 1}. 
Since $\omega_{\varepsilon}=\omega^a_{\varepsilon}+i\partial\bar{\partial}u$, we have  $(\omega_{\varepsilon}^2-(\omega^a_{\varepsilon})^2)=(\omega_{\varepsilon}+\omega^a_{\varepsilon})i\partial\bar{\partial}u=\frac{1}{2}(\omega_{\varepsilon}+\omega^a_{\varepsilon})dd^cu$, where $d=\partial+\bar{\partial}$ as usual and $d^c=i(\bar{\partial}-\partial)$. By taking integration by parts, we have
\begin{align}\label{ineq_1}
\int \zeta |\nabla u|^2(\omega^a_{\varepsilon})^2\leq -2\Big[\int (u-\mu)d\zeta\wedge d^cu\wedge (\omega_{\varepsilon}+\omega^a_{\varepsilon})+2\int\zeta(u-\mu)(e^{f_{\varepsilon}}-1)(\omega^a_{\varepsilon})^2\Big], 
\end{align}
for any compactly supported smooth function $\zeta$ and $\mu\in\mathbb{R}$. (\ref{TY_est}) tells us that
\begin{align}\label{TY}
\sup |\nabla^2u|\leq C,
\end{align}
for some constant independent of $\varepsilon$. Meanwhile, (4) in Lemma \ref{ansatz metric} tells us that $|e^{f_{\varepsilon}}-1|\leq C|f_{\varepsilon}|$.\\
\ \\
{\bf Step 2}. We have the following observation for $d^{\varepsilon}$.
\begin{lem}\label{d1_and_de}
	Let $x_0\in X$ be a fixed point not in a singular fibre and $d^{\varepsilon_0}$, $d^{\varepsilon}$ be distance functions defined by $g^a_{\varepsilon_0}$ and $g^a_{\varepsilon}$ respectively for some positive numbers $\varepsilon<\varepsilon_0$. Then there exist a constants $C,C'>0$ independent of $\varepsilon$ and $\varepsilon_0$ such that 
	\begin{align}\label{scal_dist2}
		\varepsilon^{\frac{1}{2}}d^{\varepsilon}(x,x_0)\geq {C}{(\epsilon_0)}^{\frac{1}{2}}
		d^{\epsilon_0}(x,x_0),
	\end{align} 
	for all $x$ with $d^{\varepsilon_0}(\pi(x),\pi(x_0))>C'$ and $\varepsilon<\epsilon_0$.
Furthermore, we have the reversed estimate
\begin{align}\label{scal_dist}
	\varepsilon^{\frac{1}{2}} d^{\varepsilon}(x,x_0)\leq  \frac{1}{C}{(\epsilon_0)}^{\frac{1}{2}}
	d^{\epsilon_0}(x,x_0),
\end{align}  for all $\varepsilon<\epsilon_0$.
\end{lem}
\begin{proof}
	First, we assume the case $\mathbb{B}=0$ and $s:B\rightarrow X$ is the holomorphic section for the semi-flat metric. To show the first half of the Lemma, we first observe the following:
  \begin{align}\label{2}
  	 & d^{\epsilon}(s\pi(x),s\pi(x_0))=O\big( (\varepsilon_0/\epsilon)^{\frac{1}{2}} d^{\varepsilon_0}(s\pi(x),s\pi(x_0))  \big ) \mbox{ and } d^{\varepsilon_0}(s\pi(x),s\pi(x_0))=O(|\log{|z|}|^{\frac{3}{2}})
  \end{align}   from the explicit expression \eqref{sf metric}, where $z$ is the coordinate for $\pi(x)$. On the other hand, we also have the estimates for the diameters of fibres 
\begin{align}
	 diam_{\epsilon}(L_{\pi(x)})=O(\epsilon^{\frac{1}{2}}|\log{|z|}|^{\frac{1}{2}}), |z|>2, 
\end{align} if $x$ is away from any singular fibres and 
\begin{align} \label{2'}
	  diam_{\epsilon}(L_{\pi(x)})\lesssim O((\epsilon\log{\epsilon}^{-1})^{\frac{1}{2}})
\end{align} if $x$ is near an $I_1$-singular fibre from \cite[Proposition 3.5]{GW}.
Then the lemma follows from the triangle inequality 
\begin{align*}
	d^{\epsilon}(x,x_0)\geq -d^{\epsilon}(x,s\pi(x))+d^{\epsilon}(s\pi(x),s\pi(x_0))-d^{\epsilon}(s\pi(x_0,x_0))
\end{align*}
 and \eqref{2}\eqref{2'}. 
To see the second part of the lemma when $\mathbb{B}=0$, observe that the limiting metric ${\varepsilon}\omega_{\varepsilon}^a, \varepsilon\rightarrow 0$ exits and \eqref{scal_dist} holds when $\epsilon=0$. Notice that \eqref{scal_dist} also holds when $\pi(x)=\pi(x_0)$ for all $\varepsilon>0$. Thus, there exists a tubular neighborhood of $\pi^{-1}(x_0)$ such that \eqref{scal_dist} holds for all $\varepsilon$ if $x$ falls in this tubular neighborhood. 
Then the lemma follows from the triangle inequality 
\begin{align*}
	d^{\epsilon}(x,x_0)\leq d^{\epsilon}(x,s\pi(x))+d^{\epsilon}(s\pi(x),s\pi(x_0))+d^{\epsilon}(s\pi(x_0,x_0))
\end{align*}
and \eqref{2}\eqref{2'}.

When $\mathbb{B}\neq 0$, one can represent $\mathbb{B}$ by $(U_i,\sigma_i)$ with the index set finite. Consider the geodesic on $B$ connecting $\pi(x_0), \pi(x)$. Choose a point $x_j\in U_{i_j}$, where the path intersects with $U_{i_j}$ such that the geodesic connecting $x_j,x_{j+1}$ is contained in $U_{i_j}$ after replace $\{U_i\}$ by a refinement. Then the triangle inequality implies that
\begin{align*}
\sqrt{\frac{\epsilon_0}{\epsilon}}\sum_j d^{\epsilon_0}(\sigma_{i_j}\pi(x_j),\sigma_{i_j}\pi(x_{j+1}))&-\sum_j diam_{\epsilon}(L_{\pi(x_j)})\leq	d^{\epsilon}(x,x_0)\\
&\leq \sqrt{\frac{\epsilon_0}{\epsilon}}\sum_j d^{\epsilon_0}(\sigma_{i_j}\pi(x_j),\sigma_{i_j}\pi(x_{j+1}))+\sum_j diam_{\epsilon}(L_{\pi(x_j)})
\end{align*} and the rest of the proof is similar. 

\end{proof}
By using Lemma \ref{d1_and_de}, we obtain the following Corollary.
\begin{cor}\label{two_side_distineq}
Let $x_0$ be the point given by Definition \ref{CYL}. Then there exists a constant $\tilde{C}$ such that for any point $x\in X\setminus B_{\tilde{C}}(x_0)$ and $\varepsilon<\epsilon_0$, we have
\begin{align}
d^{\varepsilon}(x,x_0)=O\big(\sqrt{\frac{\epsilon_0}{\varepsilon}}\big)d^{\epsilon_0}(x,x_0).
\end{align}
\end{cor}

In the rest of this paper, we will use $B^{d^{\varepsilon}}(x,r)=B^{d^{\varepsilon}}_r(x)$ to denote the open ball with radius $r$ centered at $x$ with respect to the distance $d^{\varepsilon}$. We will also use $d^1$ to denote the distance function
\begin{align*}
\sqrt{\varepsilon_0}d^{\varepsilon_0}
\end{align*}
and $B_r(x)$ (or $B(x,r)$) to denote $B^{d^{1}}_r(x)$. Meanwhile, the constant $\varepsilon_0$ can be chosen to be a fixed number in the rest of this paper, so we can assume it is an universal constant without loss of generality.\\
\ \\
{\bf Step 3}. Denote by $\tilde{X}$ the set $X\setminus B_{\tilde{C}}(x_0)$ in the following argument. Let $\{r_i\}$ be a non-negative increasing sequence and
\begin{align}
E_i&:=\{d^1 (x,x_0)\geq\sqrt{\varepsilon}r_{i+1}\}\cap \tilde{X},\\
A_i&:=\{\sqrt{\varepsilon}r_{i}\leq d^1(x,x_0)\leq\sqrt{\varepsilon}r_{i+1}\}\cap\tilde{X}.
\end{align}  
Then we have
\begin{align}\label{Est_of_balls}
\Big|\int_{A_i}(\omega^a_{\varepsilon})^2\Big|\leq |B^{d^1}(x_0,\sqrt{\varepsilon} r_{i+1})|\leq C\varepsilon^{\frac{2}{3}} r_{i+1}^{\frac{4}{3}}.
\end{align} Here the last inequality is due to \cite[Theorem 1.5 (iii)]{Hein}.
By (\ref{TY}) and (\ref{Est_of_balls}), one can choose $\zeta$ be a cut-off function with value 1 on $E_{i}$ and value 0 on $E_{i-1}^c$ in  (\ref{ineq_1}) such that
\begin{align}\label{ineq_2}
\int_{E_i}|\nabla u|^2\leq C(r_{i+1}-r_i)^{-1}\|u-\mu\|_{L^2(A_i)}\|\nabla u\|_{L^2(A_i)}+C\varepsilon^{\frac{2}{3}}\sum_{j=i}^{\infty}r_{j+1}^{\frac{4}{3}}\|f_{\varepsilon}\|_{L^{\infty}(A_j)}.
\end{align}
We define $\rho_i:=\frac{1}{2}(r_{i+1}+r_i)$. Suppose that
\begin{align}\label{asp1}
\rho_i-\frac{1}{C}\rho_i^{\frac{1}{3}}\leq r_i\leq r_{i+1}\leq \rho_i+\frac{1}{C}\rho_i^{\frac{1}{3}}.
\end{align}
Because $X$ is $CYL(\frac{1}{3})$ with respect to $d^1$, together with \eqref{scal_dist}\eqref{asp1} we can prove that
\begin{align}\label{domain_incl}
A_i\subset A'_i\subset B_i
\end{align}
where
\begin{align}
A'_i&:=\bigg\{x\in \tilde{X}\ \bigg|\ \sqrt{\varepsilon}\rho_i-\frac{1}{C}(\sqrt{\varepsilon}\rho_i)^{\frac{1}{3}}\leq d^{1}(x,x_0)\leq \sqrt{\varepsilon}\rho_i+\frac{1}{C}(\sqrt{\varepsilon}\rho_i)^{\frac{1}{3}}\bigg\};\\
B_i&:=B(x_i,C\varepsilon^{-\frac{1}{3}}\rho_i^{\frac{1}{3}})\mbox{ for some }x_i\in A'_{i}.
\end{align}
To see this, one should notice that Definition \ref{CYL}, Lemma \ref{d1_and_de} and Corollary \ref{two_side_distineq} imply that
\begin{align}
A_i&= \big\{x\in\tilde{X}\ \big| 
\ \sqrt{\varepsilon} r_i\leq
   d^1(x,x_0)\leq \sqrt{\varepsilon} r_{i+1}\big\}\nonumber\\
   &\subseteq\bigg\{x\in\tilde{X}\ \bigg|\  
\sqrt{\varepsilon}\rho_i-\frac{\varepsilon^{\frac{1}{3}}}{C}(\sqrt{\varepsilon}\rho_i)^{\frac{1}{3}}\leq
   d^1(x,x_0)\leq\sqrt{\varepsilon}\rho_i+\frac{\varepsilon^{\frac{1}{3}}}{C}(\sqrt{\varepsilon}\rho_i)^{\frac{1}{3}}\bigg\}\nonumber\\
  &\subseteq \bigg\{x\in\tilde{X}\ \bigg|\  
\sqrt{\varepsilon}\rho_i-\frac{1}{C}(\sqrt{\varepsilon}\rho_i)^{\frac{1}{3}}\leq
   d^1(x,x_0)\leq\sqrt{\varepsilon}\rho_i+\frac{1}{C}(\sqrt{\varepsilon}\rho_i)^{\frac{1}{3}}\bigg\}\nonumber\\
   &\subseteq B^{d^1}(x_i,(\sqrt{\varepsilon}\rho_i)^{\frac{1}{3}})\subset B(x_i,C \varepsilon^{-\frac{1}{3}}\rho_i^{\frac{1}{3}}).\nonumber
\end{align}

By taking $r_i=O(i^{\frac{3}{2}})$, (\ref{asp1}) holds for the constant $C$. We can check directly by the triangle inequality that there exists $K\in 2\mathbb{N}$ which is independent of $\varepsilon$ such that
\begin{align}
B_i\subset \cup_{j=i-\frac{K}{2}}^{i+\frac{K}{2}}A_j:=A''_i.
\end{align}
In addition, we can only consider $r_i>\frac{1}{C\sqrt{\varepsilon}}$ because otherwise $A_i=\emptyset$. Then by standard Poincar\'e inequality, (\ref{domain_incl}) implies
\begin{align}
\int_{A_i}|u-\mu|^2\leq \int_{B_i}|u-\mu|^2\leq C\varepsilon^{-\frac{2}{3}}\rho_i^{\frac{2}{3}}\int_{B_i}|\nabla u|^2\leq C\varepsilon^{-\frac{2}{3}}\rho_i^{\frac{2}{3}}\int_{A''_{i}}|\nabla u|^2
\end{align}
where $\mu=\frac{1}{|B_i|}\int_{B_i}u$. Notice that the constant in the Poincar\'e inequality only depends on the dimension of the geometry and the lower bound of the Ricci curvature. Then by the mean value theorem, we have $(r_{i+1}-r_i)^{-1}\leq C i^{-\frac{1}{2}}$ and $\rho_i^{\frac{1}{3}}\leq C(i+1)^{\frac{1}{2}}$. So (\ref{domain_incl}), (\ref{ineq_2}) and Corollary \ref{est_of_f} imply that 
\begin{align}
\int_{E_{i}}|\nabla u|^2\leq C\varepsilon^{-\frac{2}{3}}
\int_{A''_{i}}|\nabla u|^2+C\varepsilon^{\frac{2}{3}}\sum_{j=i}^{\infty}(j+1)^2e^{-\varepsilon^{-\frac{1}{6}}(j+1)}
\leq C\varepsilon^{-\frac{2}{3}}\int_{A''_{i}}|\nabla u|^2+C\varepsilon^{\frac{2}{3}}
e^{-\sigma i}
\end{align}
for $\sigma=1/2C\varepsilon^{\frac{1}{6}}$. So by taking $i=(m-\frac{1}{2}) K$, we have
\begin{align}
\int_{E_{mK}}|\nabla u|^2\leq C\varepsilon^{-\frac{2}{3}}\Bigg(\int_{E_{(m-1)K}}|\nabla u|^2-\int_{E_{mK}}|\nabla u|^2\Bigg)+C\varepsilon^{\frac{2}{3}}
e^{-\sigma (m-1)K}
\end{align}
for all $m\in\mathbb{N}$, which implies
\begin{align}
(C\varepsilon^{-\frac{2}{3}}+1)\int_{E_{mK}}|\nabla u|^2\leq C\varepsilon^{-\frac{2}{3}} \int_{E_{(m-1)K}}|\nabla u|^2 +C\varepsilon^{\frac{2}{3}}
e^{-\sigma (m-1)K}.
\end{align}
By taking $R=\Big(\frac{C+\varepsilon^{\frac{2}{3}}}{C}\Big)$,
\begin{align}
R^{m}\int_{E_{mK}}|\nabla u|^2- R^{m-1}\int_{E_{(m-1)K}}|\nabla u|^2 \leq \varepsilon^{\frac{4}{3}} R^{m-1}e^{-\sigma (m-1)K}.
\end{align}
This implies
\begin{align}
R^{m}\int_{E_{mK}}|\nabla u|^2\leq \varepsilon^{\frac{4}{3}}
\sum_{j=1}^mR^{j-1}e^{-\sigma(j-1)K}\leq \varepsilon^{\frac{4}{3}}
\sum_{j=1}^{\infty}R^{j-1}e^{-\sigma(j-1)K}.
\end{align}
When $\varepsilon$ small enough, we can take $1<R<e^{\frac{1}{2}\sigma K}$ such that \begin{align}
\int_{E_{mK}}|\nabla u|^2\leq C\varepsilon^{\frac{4}{3}}
 e^{-\delta m}
\end{align}
for all $m\in\mathbb{N}$ where 
\begin{align}
\delta=\ln R=\ln \Big(\frac{C+\varepsilon^{\frac{2}{3}}}{C}\Big)=O(\varepsilon^{\frac{2}{3}}).
\end{align}
This implies that there exists $C>0$ such that
\begin{align}\label{outside_est}
\int_{E_{r}}|\nabla u|^2\leq C\varepsilon^{\frac{4}{3}}
 e^{-\delta r^{\frac{2}{3}}}
\end{align}
for all $r>\frac{1}{C\sqrt{\varepsilon}}$.

Meanwhile, by using (\ref{ineq_1}), one can take $\zeta$ to be a cut-off function with value 1 on $B(x_0,\frac{C}{\sqrt{\varepsilon}})$ and value 0 on outside $B(x_0,\frac{C}{\sqrt{\varepsilon}}+1)$, then we have 
\begin{align}\label{ineq_2*}
\int_{B(x_0,\frac{C}{\sqrt{\varepsilon}})}|\nabla u|^2\leq C\|u-\mu\|_{L^2\big(A\big(x_0,\frac{C}{\sqrt{\varepsilon}},\frac{C}{\sqrt{\varepsilon}}+1\big)\big)}\|\nabla u\|_{L^2\big(A\big(x_0,\frac{C}{\sqrt{\varepsilon}},\frac{C}{\sqrt{\varepsilon}}+1\big)\big)}+C\|f_{\varepsilon}\|_{L^{\infty}(B(x_0,\frac{C}{\sqrt{\varepsilon}}))}.
\end{align}
Now, by Poincar\'e inequality for annular domains \cite[Lemma 2.2]{Hein}, we have
\begin{align}
\|u-\mu\|_{L^2\big(A\big(x_0,\frac{C}{\sqrt{\varepsilon}},\frac{C}{\sqrt{\varepsilon}}+1\big)\big)}\leq C\|\nabla u\|_{L^2\big(A\big(x_0,\frac{C}{\sqrt{\varepsilon}}-1,\frac{C}{\sqrt{\varepsilon}}+2\big)\big)}
\end{align}
and by Corollary \ref{est_of_f}, we have
\begin{align}
\|f_{\varepsilon}\|_{L^{\infty}(B(x_0,\frac{C}{\sqrt{\varepsilon}}))}\leq \|f_{\varepsilon}\|_{L^{\infty}\big(A\big(x_0,\frac{1}{C\sqrt{\varepsilon}},\frac{C}{\sqrt{\varepsilon}}\big)\big)}\leq C\varepsilon^{\frac{4}{3}}.
\end{align}
In fact, Corollary \ref{est_of_f} implies the middle term of this inequality decays exponentially with respect to $\varepsilon$. Therefore, it is bounded by $C\varepsilon^q$ for any $q>0$ and $C$ depending on $q$.
So (\ref{ineq_2*}) implies
\begin{align}\label{inside_est}
\int_{B(x_0,\frac{C}{\sqrt{\varepsilon}})}|\nabla u|^2\leq C\|\nabla u\|_{L^2\big(A\big(x_0,\frac{C}{\sqrt{\varepsilon}}-1,\frac{C}{\sqrt{\varepsilon}}+2\big)\big)}+C\varepsilon^{\frac{4}{3}}\leq C\varepsilon^{\frac{4}{3}}
\end{align}
when $\varepsilon$ small enough. Combine (\ref{outside_est}) and (\ref{inside_est}), we obtain
\begin{align}\label{ineq_3}
\int_{X}|\nabla u|^2\leq C\varepsilon^{\frac{4}{3}}.
\end{align}
\end{proof}
Notice that Hein concluded the exponential decay for the square integral of $|\nabla u_{\varepsilon}|$ \cite{Hein}, which is not our main concern. We are more interested in the vanishing order of the coefficient in front of it. However, the exponential decay plays an important rule here when we consider the $L^2$-estimate of $|\nabla u_{\varepsilon}|$ on the complement of a compact region.

By using Proposition \ref{lema_Hein}, we are now ready to prove Theorem \ref{c2alpha}.
\begin{proof}[Proof of Theorem \ref{c2alpha} ]
Let $R>0$ be any positive number. By Lemma \ref{d1_and_de}, we have
\begin{align}
B^{d^{\varepsilon}}(x_0,\varepsilon^{-\frac{1}{2}} R)\supset B^{d^{\varepsilon_0}}(x_0, R)
\end{align}

Let $\mu_0=\int_{B^{d^{\varepsilon}}(x_0,\varepsilon^{-\frac{1}{2}} R)} u$. We can apply the standard Nash-Moser iteration on the equation $L(u-\mu_0)=(e^{f_{\varepsilon}}-1)$ defined on $B^{d^{\varepsilon}}(x_0,\varepsilon^{-\frac{1}{2}} R)$. Then we have
\begin{align}\label{moser}
\|u-\mu_0\|_{L^{\infty}(B^{d^{\varepsilon}}(x_0,\varepsilon^{-\frac{1}{2}} R)}\leq C \|u-\mu_0\|_{L^2(B^{d^{\varepsilon}}(x_0,2\varepsilon^{-\frac{1}{2}} R))},
\end{align}
for some $C$ depending only on the upper bound of $\frac{\Lambda}{\lambda}$, which is universal because of (\ref{L_over_l}). Meanwhile, by using the Poincar\'e inequality \cite[Lemma 2.1]{Hein} and (\ref{ineq_3}), the right-hand side of (\ref{moser}) can be bounded:
\begin{align}\label{l2est}
\|u-\mu_0\|_{L^2(B^{d^{\varepsilon}}(x_0,\varepsilon^{-\frac{1}{2}} R))}&\leq C\varepsilon^{-\frac{1}{2}} R\|\nabla u\|_{L^2(B^{d^{\varepsilon}}(x_0,2\varepsilon^{-\frac{1}{2}} R))}\\
&\leq C\varepsilon^{\frac{2}{3}}
 \varepsilon^{-\frac{1}{2}} R=C\varepsilon^{\frac{1}{6}}R.\nonumber
\end{align}
So this estimate holds for a smaller domain $B^{d^{\varepsilon_0}}(x_0, R)$. So this estimate holds for a smaller domain $B^{d^{\varepsilon_0}}(x_0, R)$. Combine (\ref{l2est}) and (\ref{moser}), we have
\begin{align}\label{estpt1}
\|u-\mu_0\|_{C^0_{g^a_{\varepsilon}}(B^{d^{\varepsilon_0}}(x_0, R))}=\|u-\mu_0\|_{L^{\infty}(B^{d^{\varepsilon_0}}(x_0, R))}\leq CR\varepsilon^{\frac{1}{6}}
\end{align}
Therefore by using the Schauder estimate, (\ref{estpt1}) and (5) in Lemma \ref{ansatz metric}, we obtain
\begin{align}
\Big\|\ u-\dashint_{B^{d^{\varepsilon_0}}(x_0, R)}u \ \Big\|_{C^{2,\alpha}_{g^a_{\varepsilon}}(B^{d^{\varepsilon_0}}(x_0, R))}&\leq C\Bigg(\|u-\mu_0\|_{C^0_{g^a_{\varepsilon}}(B^{d^{\varepsilon_0}}(x_0, 2R))}+\|f_{\varepsilon}\|_{C^{0,\alpha}_{g^a_{\varepsilon}}(B^{d^{\varepsilon_0}}(x_0, 2R))}\Bigg)\\
&\leq CR\varepsilon^{\frac{1}{6}}+Ce^{-1/C\varepsilon^{1-p}}\leq CR\varepsilon^{\frac{1}{6}}\nonumber
\end{align}
and we complete this proof.
\end{proof}
\begin{rk}
Although we only have $C^{2,\alpha}$-estimate for $(u_{\varepsilon}-\dashint u_{\varepsilon})$, one shall remember that the term we are concerned is $\partial\bar{\partial}u_{\varepsilon}$ appears in the Monge-Amp\'ere equation. According to Theorem \ref{c2alpha}, we have 
\begin{align}\label{alpha_est}
\|\partial\bar{\partial}u_{\varepsilon}\|_{C^{\alpha}_{g^a_{\varepsilon}}(B^{d^{\epsilon_0}}(x_0,R))}\leq CR\varepsilon^{\frac{1}{6}}.
\end{align}
\end{rk}
\begin{rk}
	Here we may also perform the hyperK\"ahler gluing of Donaldson \cite{D2} (see also \cite{Fos}\cite{HSVZ}) with the modification such that the elliptic fibration complex structure is preserved \cite{CVZ}. However, it is less clear if all Ricci-flat metrics near the collapsing limits can be derived this way as the hyperK\"ahler gluing does not preserve the cohomology classes. 
\end{rk}

\section{Collapsing Limits of $ALH^*$-Gravitational Instantons and Partial Compactification of Pointed Moduli Space of $ALH^*$-Gravitational Instantons} \label{sec: compactification}
In this section, we want to have some description of the compactification of the pointed $ALH^*$-gravitational instantons.
 For the focus of the paper in the later discussion, we have the following definition: 
\begin{defn}
	A sequence of pointed $ALH^*_d$-gravitational instantons $(\check{Y}_i,\check{D}_i,\check{c}_i,[\check{\omega}_i],p_i)$ is admissible if 
		 \begin{itemize}
			\item $\check{Y}$ is a weak del Pezzo surface of degree $d$, i.e. $(-K_{\check{Y}})^2=d$ and $-K_{\check{Y}}$ is big and nef. 
		\item $\check{D}$ is a smooth anti-canonical divisor of $\check{Y}$. By adjunction, it is a smooth elliptic curve. 
		\item $[\check{\omega}]$ is a K\"ahler class of $\check{Y}$. 
		\item $\check{c}\in \mathbb{R}_+$ determines a scaling of the model metric at its end. 
		\item $p_i\in \check{Y}_i\setminus \check{D}_i$ falls in a finite region (see Definition \ref{finite region}). 
	\end{itemize}
\end{defn}

%
%
 
  Now given an admissible sequence of pointed $ALH^*_d$-gravitational instantons, we want to study the pointed Gromov-Hausdorff limit of it. Notice that the possible $3$ or $2$-dimensional limits are classified by Sun-Zhang \cite[Theorem 4.4, Theorem 5.5]{SZ} but it is not clear given a sequence what will be the pointed Gromov-Hausdorff limit. 
 To obtain such a result, we should recall the definition of Gromov-Hausdorff convergence. Readers can see \cite{P} for more detail. Let $(X, d_1)$, $(Y, d_2)$ be two compact metric spaces. Then we can define the Gromov-Hausdorff distance to be
 \begin{align}
 	d_{GH}(X,Y)=\inf\Big\{d(X,Y)\mbox{ }\Big|\mbox{ }d&\mbox{ is a metric defined on }X\sqcup Y\\
 	&\mbox{ with }d|_{X\times X}=d_1\mbox{ and }d|_{Y\times Y}=d_2\Big\}.\nonumber
 \end{align}
 Here $d(X,Y)$ is defined in the sense of the Hausdorff distance, i.e.,
 \begin{align}
 	d(A,B):=\inf \big\{r\mbox{ }\big|\mbox{ }A\subset \cup_{x\in B}B^d(x,r)\mbox{ and } B\subset \cup_{x\in A}B^d(x,r) \big\}
 \end{align}
 for all compact sets $A,B\subset X\sqcup Y$. In general, we omit the metrics $d_1$, $d_2$ in our notation when the metrics are obvious in the content. When the metrics on sets $X$ and $Y$ are not specific, we will use the notation $d_{GH}((X,d_1),(Y,d_2))$ (or $d_{GH}((X,g_1),(Y,g_2))$) instead to clarify them. In our case, since the manifolds we consider are complete, non-compact. So we need to use the following version of the Gromov-Hausdorff convergence.
 \begin{defn}
 	For any sequence of pointed metric spaces $\{(X_i,x_i, d_i)\}_{i\in\mathbb{N}}$, we call the sequence has the Gromov-Hausdorff limit
 	\begin{align}
 		(X_i,x_i, d_i)\rightarrow (X,x_0,d_0)
 	\end{align}
 	if and only if for any $R>0$
 	\begin{align}
 		\lim_{i\rightarrow \infty}d_{GH}(\overline{B^{d_i}(x_i, R)},\overline{B^{d_0}(x_0,R)})=0.
 	\end{align}
 \end{defn}

 \begin{thm}\label{thm: compactify 1} Let $(\check{Y}_i,\check{D}_i,[\check{\omega}_t],c_i,p_i)$ be a sequence of admissible $ALH^*$-gravitational instantons.  Denote by $\tau_i$ the modulus of $\check{D}_i$ in the standard fundamental domain of $\mathbb{H}/SL(2,\mathbb{Z})$, where $\mathbb{H}$ is the upper half plane and assume that $\mbox{Im}\tau_i\rightarrow \infty$. Denote by $g^a_{\epsilon_0}(X_i)$ the Riemannian metric associate to the ansatz metric in Lemma \ref{ansatz metric} for the elliptic fibration on $X_i$. Assume further that $((f_i\circ \pi_0)^{-1}(K), \epsilon_0^{-\frac{1}{2}}g^a_{\epsilon_0}(X_i),p_i)$ converges in the pointed Gromov-Hausdorff topology for a fixed $\epsilon_0\ll 1$,
 	 then we have the following:
 	\begin{enumerate}
 			\item If $c_i(\mbox{Im}\tau_i)^{\frac{3}{2}}\rightarrow \infty$, $c_i(\mbox{Im}\tau_i)^{-1/6}\rightarrow 0$, then the pointed Gromov-Hausdorff limit is $\mathbb{R}^2$ with the flat metric.  
 		\item If $c_i(\mbox{Im}(\tau_i))^{\frac{3}{2}}\rightarrow C$ for some constant $C>0$, the pointed Gromov-Hausdorff limit is $\mathbb{R}^2$ with a scaling of the special K\"ahler metric $Wdz\otimes d\bar{z}$, where the scaling depends strictly increasingly on $C$ and $W$ which are defined in \eqref{special Kahler}. 
 		\item If $c_i(\mbox{Im}(\tau_t))^{\frac{3}{2}}\rightarrow 0$, the pointed Gromov-Hausdorff limit is $\mathbb{R}_+$.
 	\end{enumerate}
 \end{thm} 
\begin{proof}
For each $(\check{Y}_i,\check{D}_i,[\check{\omega}_i],c_i,p_i)$, we denote $\check{Y}_i\setminus \check{D}_i$ by $\check{X}_i$. By Theorem \ref{theorem:CJL}, there exists a special Lagrangian fibration on $\check{X}_i$, with the fibre class determined by the vanishing cycle in $H_1(\check{D}_i,\mathbb{Z})$. Moreover, the suitable hyperK\"ahler rotation $X_i$ can be compactified to a rational elliptic surface $Y_i$ by adding an $I_d$-fibre $D_i$ at its end again by Theorem \ref{theorem:CJL}. Denote by $\omega_i$ the Ricci-flat metric on $X_i$ from hyperK\"ahler rotation. From Theorem \ref{thm: local model}, we have $\omega_i$ is asymptotic to $\alpha_i\omega_{sf,b_{0,i},\epsilon_i}$, where 
 \begin{align*}
 	\epsilon_i=\frac{2\sqrt{2}\pi}{\mbox{Im}(\tau_i)}\rightarrow 0, \hspace{4mm}
 	\alpha_i=d^{\frac{1}{2}} 2^{-\frac{3}{4}}\pi\epsilon_i^{-\frac{1}{2}}, 
 \end{align*} and $b_{0,i}$ is bounded with bounds independent of $i$. From the Torelli theorem of $ALH^*$-gravitational instantons (Theorem \ref{torelli}), there exists a unique Ricci-flat metric in $[\omega_i]$ with the above asymptotic. 
Since hyperK\"ahler rotations do not change the underlying spaces, $p_i\in \check{X}_i$ naturally gives $p_i\in X_i$. To sum up, the given data $(\check{Y}_i,\check{D}_i,[\check{\omega}_i],c_i,p_i)$, as Riemannian manifolds, are the same as a sequence of $(Y_i,D_i,[\omega_i],\alpha_i, p_t)$, where 
\begin{itemize}
	\item $Y_i$ is a rational elliptic surface with an $I_d$-fibre $D_t$,
	\item $[\omega_i]\in H^2(X_i,\mathbb{R})$ a K\"ahler class, $X_i=Y_i\setminus D_i$.
	\item $\alpha_i\in \mathbb{R}_+$.
	\item $p_i\in X_i$ falls in the finite region and $p_i\rightarrow p_0\in Y_0$ on a smooth fibre. 
\end{itemize} From Theorem \ref{LLL} and Remark \ref{rk LLL}, we have $Y_i\rightarrow Y_0$ to a fixed rational elliptic surface with an distinguished (if more than one) $I_d$-fibre $D_0$. Denote $Y_0\setminus D_0$ by $X_0$. 
%

	Under the assumption $\mbox{Im}(\tau_i)\gg 0$, we have the Calabi-Yau metrics are close to the hyperK\"ahler rotation of certain scaling of semi-flat metrics (up to twisting of B-fields), which is flat along the smooth fibres away from the singular ones. Denote by $g_i$ the Riemannian metric on $\check{X}_i\cong X_i$.
	From Corollary \ref{cor: limit}, we have 
	 \begin{align*}
	 	\| \sigma_i (\omega_i-\mathbb{B}\omega^a_{\epsilon_i})\|_{C^{0,\alpha}_{\sigma_i g^a_{\epsilon_i}}(B^{\sigma g^a_{\epsilon_i}(X_i)}(x_i,R)))}= 	\|  \omega_i-\mathbb{B}\omega^a_{\epsilon_i}\|_{C^{0,\alpha}_{ g^a_{\epsilon_i}}(B^{ g^a_{\epsilon_i}(X_i)}(x_i,\sigma_i^{-1/2} R)))}\leq O(\sigma_i^{-\frac{1}{2}}R\epsilon_i^{\frac{2}{3}}). 
	 \end{align*} Equivalently, in terms of Gromov-Hausdorff distance, we have 
	 \begin{align}\label{eq: GH conv}
	 	 d_{GH}\bigg(  B^{\sigma_i g_i}(p_i,R), B^{ \sigma_i g^a_{\epsilon_i}(X_i)}(p_i,R) \bigg)\rightarrow 0 \mbox{ if } \sigma_i^{-\frac{1}{2}}\epsilon_i^{\frac{2}{3}}\rightarrow 0.
	 \end{align}
	For our purpose, we will always have  $\sigma_i=O(c_i(\mbox{Im}\tau_i)^{\frac{1}{2}})$ from the hyperK\"ahler relation in Theorem \ref{thm: local model}.
  We will first show (2), where we take $\sigma_i=O(\epsilon_i)$ or equivalently $c_i=O((\mbox{Im}\tau_i)^{-\frac{3}{2}})$. 
From Lemma \ref{ansatz metric} (7) and the explicit form of semi-flat metric, we have a uniform convergence
     \begin{align} \label{eq: GH collapsing}
     	d_{GH}\bigg((B^{\epsilon g^a_{\epsilon}(X_i)}(p_i,R), \epsilon g^a_{\epsilon}(X_i)), (\pi_i(B^{ \epsilon g^a_{\epsilon}(X_i)}(p_i,R)),\omega^i_{SK})\bigg)\rightarrow 0
    \end{align}	as $\epsilon \rightarrow 0$. 
    From the discussion in Section \ref{sec: synergy}, there exists $f_i:B_0\rightarrow B_i$ such that 
    \begin{align}\label{eq: GH base}
    	  d_{GH}((K_i,f_i^*\omega^i_{SK}), (K_i, \omega^0_{SK}))\rightarrow 0,
    \end{align} where $K_i$ is a complement of discs of diameter $(\mbox{Im}\tau_i)^{-1/2}\rightarrow 0$ (with respect to $\omega^0_{SK}$). 
    Then (2) follows from the diagonal sequence argument together with \eqref{eq: GH conv}\eqref{eq: GH collapsing} and \eqref{eq: GH base}.

    Now assuming $c_i(\mbox{Im}(\tau_i))^{3/2}\rightarrow \infty$ and $c_i (\mbox{Im}(\tau_i))^{-1/6}\rightarrow 0$. The later implies that \eqref{eq: GH conv} is valid and together with the explicit form of semi-flat metrics implies that the elliptic fibration is collapsing because $c_i\mbox{Im}(\tau_i)^{-1/2}\rightarrow 0$ from the assumption. Then the first assumption $c_i(\mbox{Im}(\tau_i))^{3/2}\rightarrow \infty$ implies that the Gromov-Hausdorff limit is the tangent plane of $p_0$ and this shows (1).

   Finally, (3) comes from the same argument that the tangent cone of the $ALH^*$-gravitational instantons, which is $\mathbb{R}_+$ \cite[Theorem 1.5]{Hein}. From \eqref{eq: GH conv}, we may study the Gromov-Hausdorff limit by replacing $\omega_i$ by $c_i(\mbox{Im}(\tau_i))^{1/2}\omega^a_{\epsilon_i}$ on $X_i$. Similar to the above cases, $c_i(\mbox{Im}(\tau_i))^{1/2}\rightarrow 0$ implies that the elliptic fibration is collapsing. From $f_i^*\omega^i_{SK}\rightarrow \omega^0_{SK}$, the desired pointed Gromov-Hausdorffo limit is exactly the tangent cone of $\omega_{SK}$ because $c_i (\mbox{Im}\tau_i)^{3/2}\rightarrow 0$. From the explicit calculation, with respect to $\omega^0_{SK}$ the distance from a fixed point to $z_0$ is of order $|\log{|z_0|}|^{3/2}$ when $|z_0|\rightarrow 0$ and the length of the circle of $\{|z|=|z_0|\}$ is $|\log{|z_0|}|^{1/2}$. This implies that the tangent cone of the base with respect to $\omega^0_{SK}$ is $\mathbb{R}_+$. 
   
\end{proof}
\begin{rk}
	 Here the assumption $c_i(\mbox{Im}(\tau_i))^{3/2}\rightarrow C$ is the replacement for bounded diameter in the work of Gross-Wilson \cite{GW} and Lott \cite{L}. Indeed, the assumption implies that given any two points in $X$, the distances with respect to the sequence of the metrics are bounded above and below.
\end{rk}

From the geometry of the semi-flat metric, it is natural to make the following conjecture:
\begin{conj}
	Under the assumption of Theorem \ref{thm: compactify 1}, we have 
	\begin{enumerate}
		\item If $c_i(\mbox{Im}\tau_i)^{-1/2}\rightarrow \infty$, then the pointed Gromov-Hausdorff limit is $\mathbb{R}^4$ with the flat metric. 
	\item If $c_i(\mbox{Im}\tau_i)^{-1/2}\rightarrow C$ and $\mbox{Im}(\tau_i)\rightarrow \infty$ for some constant $C>0$, then the pointed Gromov-Hausdorff limit is $\mathbb{R}^2\times T^2$ with the product of flat metrics such that the size of $T^2$ is strictly increasing depending on $C$.
	\end{enumerate}
\end{conj} Notice that under the assumption above , the assumption of \eqref{eq: GH conv} does not hold. Finally, we give the following remark to conclude the paper. 
\begin{rk}
	The moduli space of metric K3 is a symmetric space 
	\begin{align*}
			SO(\Lambda_{K3})\backslash SO(3,19)/SO(3)\times SO(19),
		\end{align*} where $\Lambda_{K3}$ is the K3 lattice. The compactification of symmetric spaces are widely studied, for instance \cite{BJ}. Odaka-Oshima used the known compactifications from algebraic geometry to recover the geometric properties of the K3 metrics near certain boundary strata of the K3 moduli spaces \cite{OO}. From Torelli theorem of the $ALH^*$-gravitational instantons and the description of period domain, one can expect that the moduli space of $ALH^*$-gravitational instantons are also symmetric spaces. For instance, the moduli space of $ALH^*_9$-gravitational instantons is 
	\begin{align*}
			SL(2,\mathbb{Z})\backslash \mathcal{H}\times \mathbb{R}_+\cong SL(2,\mathbb{R})/SO(2)\times \mathbb{R}_+,   
	\end{align*} again a  symmetric space, where $\mathcal{H}$ is the upper half plane. One may follow the similar idea of Odaka-Oshima to understand the behavior of the metrics near the boundary of the moduli spaces.

\end{rk}

\end{document}